\RequirePackage[l2tabu, orthodox]{nag}

\documentclass[12pt, reqno]{amsart}
\usepackage{url,amssymb,enumerate,colonequals}
\usepackage[margin = 1in]{geometry}
\usepackage{pdfsync}
\usepackage{multirow}
\usepackage[usenames,dvipsnames]{xcolor}
\usepackage{url}

\usepackage{color}

\newcommand{\OO}{\mathcal{O}}
\newcommand{\F}{\mathbb{F}}

\newcommand{\Q}{\mathbb{Q}}

\newcommand{\Z}{\mathbb{Z}}

\newcommand{\pp}{\mathfrak{p}}

\newcommand{\calO}{\mathcal{O}}

\DeclareMathOperator{\Aut}{Aut}

\DeclareMathOperator{\End}{End}

\DeclareMathOperator{\Hom}{Hom}

\DeclareMathOperator{\lcm}{lcm}

\DeclareMathOperator{\norm}{norm}

\DeclareMathOperator{\N}{N}

\DeclareMathOperator{\Spec}{Spec}

\DeclareMathOperator{\tr}{tr}

\DeclareMathOperator{\Trd}{Trd}
\DeclareMathOperator{\Nrd}{Nrd}
\newcommand{\QQ}{\mathbb{Q}}
\DeclareMathOperator{\Fqbar}{\overline{\F}_q}

\DeclareMathOperator{\Def}{\overset{{}_{\text{def}}}{=}}

\newcommand{\isom}{\simeq}

\newtheorem{theorem}{Theorem}[section]
\newtheorem{lemma}[theorem]{Lemma}
\newtheorem{corollary}[theorem]{Corollary}
\newtheorem{proposition}[theorem]{Proposition}

\theoremstyle{definition}

\newtheorem{definition}[theorem]{Definition}

\newtheorem{example}[theorem]{Example}

\newtheorem{remark}[theorem]{Remark}

\usepackage{microtype}

\numberwithin{equation}{subsection}

\usepackage{tikz}
\usepackage{float}
\usetikzlibrary{calc}
\usetikzlibrary{backgrounds}
\usetikzlibrary{arrows,%
	petri,%
	topaths}%
\usepackage{pgf}
\usetikzlibrary{arrows,automata}
\usepackage[position=top]{subfig}
\usepackage[latin1]{inputenc}
\newcommand\DoubleLine[7][4pt]{%
	\path(#2)--(#3)coordinate[at start](h1)coordinate[at end](h2);
	\draw[#4]($(h1)!#1!90:(h2)$)-- node [auto=left] {#5} ($(h2)!#1!-90:(h1)$); 
	\draw[#6]($(h1)!#1!-90:(h2)$)-- node [auto=right] {#7} ($(h2)!#1!90:(h1)$);
}

\usetikzlibrary{positioning,calc,intersections}

\begin{document}
%%%%%%%%%%%%%%%%%%%%%%%%%%%%%%%%%%%%%%%%%%%%%%%5TITLE STUFF%%%%%%%%%%%%%%%%%%%%%%%%%%%%%%%%%%%%%%%%%%%%%%%%%%%%%%%%%%%%%%%%%%%%%	
	\title[Cycles in the supersingular $\ell$-isogeny graph]{Cycles in the supersingular $\ell$-isogeny graph and
          corresponding endomorphisms}
	
	\author[Bank]{Efrat Bank}
	\address{Department of Mathematics, University of Michigan, Ann Arbor, MI, USA}
	\email{ebank@umich.edu}
	\urladdr{\url{http://www-personal.umich.edu/~ebank/}}
	
	\author[Camacho-Navarro]{Catalina Camacho-Navarro}
	\address{Department of Mathematics\\ Colorado State University\\Fort Collins, CO, 80523, USA}
	\email{camacho@math.colostate.edu}
	%\urladdr{\url{http://}}
	
	\author[Eisentr\"ager]{Kirsten Eisentr\"ager}
\address{Department of Mathematics \\ The Pennsylvania State
  University \\ University, Park, PA 16802, USA}
  \email{eisentra@math.psu.edu}
  \urladdr{http://www.personal.psu.edu/kxe8/}	
	
	\author[Morrison]{Travis Morrison}
\address{Institute for Quantum Computing \\ The University of Waterloo
  \\ Waterloo, ON }
  \email{travis.morrison@uwaterloo.ca}
% \urladdr{http://www.personal.psu.edu/txm950/}

	\author[Park]{Jennifer Park}
	\address{Department of Mathematics, University of Michigan, Ann Arbor, MI, USA}
	\email{jmypark@umich.edu}
	\urladdr{\url{http://www-personal.umich.edu/~jmypark/}}
	
\maketitle	
\begin{abstract}
  We study the problem of generating the endomorphism ring of a
  supersingular elliptic curve by two cycles in $\ell$-isogeny
  graphs. We prove a necessary and sufficient condition for the two
  endomorphisms corresponding to two cycles to be linearly
  independent, expanding on the work by Kohel in his thesis. We also
  give a criterion under which the ring generated by two cycles is
  not a maximal order. We give some examples in which we compute
  cycles which generate the full endomorphism ring. The most difficult
  part of these computations is the calculation of the trace of these
  cycles. We show that a generalization of Schoof's algorithm can
  accomplish this computation efficiently.
\end{abstract}

%%%%%%%%%%%%%%%%%%%%%%%%%%%%%%%%%%%%%%%%%%%%%%%%%%%%%%%%%%%%%%%%%%%%%%%%%%%%%%%%%%%%%%%%%%%%%%%%%%%%%%%%%%%%%	
\section{Introduction}\label{sec: introduction}	
%%%%%%%%%%%%%%%%%%%%%%%%%%%%%%%%%%%%%%%%%%%%%%%%%%%%%%%%%%%%%%%%%%%%%%%%%%%%%%%%%%%%%%%%%%%%%%%%%%%%%%%%%%%%%
The currently used cryptosystems, such as RSA and systems based on
Elliptic Curve Cryptography (ECC), are known to be broken by quantum
computers. However, it is not known whether cryptosystems based on the
hardness of computing endomorphism rings or isogenies between
supersingular elliptic curves can be broken by quantum computers.
Because of this, these systems have been studied intensely over the
last few years, and the International Post-Quantum Cryptography
Competition sponsored by NIST~\cite{PQCS} has further increased
interest in studying the security of these systems. There is a
submission under consideration~\cite{SIKE} which is based on
supersingular isogenies.

Cryptographic applications based on the hardness of computing
isogenies between supersingular elliptic curves were first given
in~\cite{CLG2009}. In this paper, Charles, Goren, and Lauter constructed a
hash function from the $\ell$-isogeny graph of supersingular elliptic
curves, and finding preimages for the hash function is connected to
finding certain $\ell$-power isogenies (for a small prime $\ell$)
between supersingular elliptic curves. 

More recently, De Feo, Jao, and Pl\^ut~\cite{DFJLP14} proposed
post-quantum key-exchange and encryption schemes based on
computing isogenies of supersingular elliptic curves.  A signature scheme
based on supersingular isogenies is
given in~\cite{YAJJS}, and~\cite{GPS2016} gives a signature scheme in
which the 
secret key is a maximal order isomorphic to the endomorphism ring of a
supersingular elliptic curve.

There are currently no subexponential classical or quantum attacks for
these systems.  However, under some heuristic assumptions, the
quaternion analogue for the underlying hardness assumption of the hash
function in~\cite{CLG2009} was broken in~\cite{KLPT}, which suggests
that a careful study of the isogenies and endomorphism rings of
supersingular elliptic curves is necessary.

For a fixed $q$, \cite{Cer14, ML04} list all isomorphism classes of
supersingular elliptic curves over $\F_q$ along with their maximal
orders in a quaternion algebra (which was improved in \cite[\S
5.2]{CG14}). In~\cite{Mur14} McMurdy also computes explicit
endomorphism rings for some supersingular elliptic curves. The problem of computing isogenies between supersingular elliptic
curves over $p$ has been studied, both in the classical
setting~\cite[Section 4]{GD16} where the complexity of the algorithm
is $\tilde{O}(p^{1/2})$, and in the quantum setting~\cite{BJS14},
where the complexity is $\tilde{O}(p^{1/4})$. In fact, computing the
endomorphism ring of a supersingular elliptic curve is deeply
connected to computing isogenies between
supersingular elliptic curves, as shown by \cite{Kohel}. Heuristic
arguments show that these two problems are equivalent~\cite{GPS2016,KLPT,EHLMP}.
%Heuristically, these two problems should be
%equivalent~\cite{GPS2016}, \cite{KLPT}, \cite{EHLMP}. %In
%\cite{EHLMP}, they are shown to be equivalent under some heuristic
%assumptions about distributions of certain norm forms and on the
%isogeny graph. %Algorithms for related problems appear in other
%work. In~\cite{GPS2016} there is an algorithm for finding power-smooth
%degree isogenies.  

In this paper, we work over finite fields $\F_q$ of characteristic
$p$, and study the problem of generating the endomorphism ring of a
supersingular elliptic curve $E$ by two cycles in the {\em
  $\ell$-isogeny graph} (Definition \ref{deg: isogeny graph}) of
supersingular elliptic curves. Computing the endomorphism ring of a
supersingular elliptic curve via $\ell$-isogeny graphs was first
studied by Kohel~\cite[Theorem 75]{Kohel}, who gave an approach for
finding four linearly independent endomorphisms that generate a
finite-index suborder of $\End(E)$ by finding cycles in the
$\ell$-isogeny graph. %of supersingular elliptic curves,
The running time of the probabilistic algorithm is $O(p^{1 +
  \varepsilon})$. We demonstrate some obstructions to generating the
full endomorphism ring with two cycles $\alpha, \beta \in \End(E)$.

Expanding on \cite{Kohel}, we prove in Theorem~\ref{thm:independent
  endos} the necessary and sufficient conditions for $\alpha$ and $\beta$
to be linearly independent. We also prove sufficient conditions for when $\alpha$
and $\beta$ generate a proper suborder of $\End(E)$ in
Theorem~\ref{thm:nonmaximal order}, then compute some examples. In
order to do this, we need to detect when the
order generated by two cycles is isomorphic to another given order; in
\S\ref{sec: endo rings computation}, we give a criterion that reduces this problem to computing traces of
various endomorphisms. In the appendix we give a generalization of
Schoof's algorithm~\cite{Schoof} (using the improvements from \cite{SS15}) and show that the trace of an
arbitrary endomorphism of norm $\ell^e$ can be computed in time
polynomial in $e, \ell$ and $\log p$.

The paper is organized as follows. In \S\ref{sec: isogeny
  graphs} we review some definitions about elliptic curves and define
isogeny graphs. In \S\ref{sec: endo rings computation} we
discuss some background on quaternion algebras, the Deuring
correspondence, and we discuss how to compute the endomorphism ring of
a supersingular elliptic curve from cycles in the supersingular $\ell$-isogeny
graph. In \S\ref{sec: sufficient conditions} we give 
a necessary and sufficient condition for two endomorphisms to be linearly independent,
expanding on a result by Kohel~\cite{Kohel}. In \S\ref{sec:
  obstructions} we give conditions under which two cycles $\alpha,
\beta$ in the supersingular $\ell$-isogeny graph generate a proper suborder of the
endomorphism ring. In \S\ref{sec: examples} we compute some
examples, and in the Appendix we give the generalization of Schoof's algorithm.

\subsection*{Acknowledgments} We are deeply grateful to John Voight
for several helpful discussions, and for providing us with some code
that became a part of the code that generated the computational
examples in Section~\ref{sec: examples}. We also thank Sean Hallgren and
Rachel Pries for helpful discussions. We thank Andrew Sutherland for
many comments on an earlier version that led to a significant improvement in the
running time analysis of the generalization of Schoof's algorithm, and for outlining the proofs of Proposition~\ref{prop:endsdefined} and Lemma~\ref{lemma:ell-torsion}. Finally,
we thank the Women in Numbers 4 conference and BIRS, for enabling us
to start this project in a productive environment. K.E. and T.M. were
partially supported by National Science Foundation awards DMS-1056703
and CNS-1617802.

%\pagebreak

%\jenn{motivate the problem, talk about the history of the work, application to isogeny graphs to cryptography}
%\subsection{Connections to cryptography}\label{sec: crypto}
%%%%%%%%%%%%%%%%%%%%%%%%%%%%%%%%%%%%%%%%%%%%%%%%%%%%%%%%%%%%%%%%%%%%%%%%%%%%%%%%%%%%%%%%%%%%%%%%%%%%%%%%%%%%%
%%%%%%%%%%%%%%%%%%%%%%%%%%%%%%%%%%%%%%%%%%%%%%%%%%%%%%%%%%%%%%%%%%%%%%%%%%%%%%%%%%%%%%%%%%%%%%%%%%%%%%%%%%%%%
\section{Isogeny graphs}\label{sec: isogeny graphs}
%%%%%%%%%%%%%%%%%%%%%%%%%%%%%%%%%%%%%%%%%%%%%%%%%%%%%%%%%%%%%%%%%%%%%%%%%%%%%%%%%%%%%%%%%%%%%%%%%%%%%%%%%%%%%
\subsection{Definitions and properties}\label{sec: defs and notation}
In this section, we recall several definitions and notation that are
used throughout. We refer the reader to \cite{AEC} and \cite{Kohel} for a detailed overview on some of the below. Let $k$ be a field of characteristic $p>3$.
%We
%begin with defining supersingular elliptic curves and the isogeny
%graphs.

By an elliptic curve $E$ over a field $k$, we
mean the projective curve with an affine model $E:y^2=x^3+Ax+B$ for some $A,B\in k$. The
points of $E$ are the points $(x,y)$ satisfying the curve equation,
together with the point at infinity. These points form an abelian
group. The $j$-invariant of an elliptic curve given as above is
$j(E)=\frac{256 \cdot 27 \cdot A^3}{4A^3 + 27B^2}$. Two elliptic
curves $E,E'$ defined over a field $k$ have the same $j$-invariant if
and only if they are isomorphic over the algebraic closure of $k$.

Let $E_1$ and $E_2$ be elliptic curves defined over $k$. An {\em isogeny} $\varphi:E_1 \to E_2$
defined over $k$ is a non-constant rational map which
is also a group homomorphism from $E_1(k)$ to
$E_2(k)$~\cite[III.4]{AEC}. The \textit{degree} of an isogeny is its degree as
a rational map. When the degree $d$ of the isogeny $\varphi$ is
coprime to $p$, then $\varphi$ is separable and every separable
isogeny of degree $d>1$ can be factored into a
composition of isogenies of prime degrees such that the product of the
degrees equals $d$. If $\psi: E_1\to E_2$ is an isogeny of degree $d$,
the {\em dual isogeny} of $\psi$ is the unique isogeny $\widehat{\psi}: E_2\to E_1$ satisfying 
$\widehat{\psi}\psi=[d]$, where $[d]:E_1\to E_1$ is the
multiplication-by-$d$ map.

In this paper we will be interested in isogenies of $\ell$-power
degree, for $\ell$ a prime different from the characteristic of $k$.
We can describe a separable isogeny from an elliptic curve $E$ to some
other elliptic curve via its kernel. Given an elliptic curve $E$ and a
finite subgroup $H$ of $E$, there is a separable isogeny
$\varphi: E\to E'$ having kernel $H$ which is unique up to
isomorphism~(see \cite[III.4.12]{AEC}). In this paper we will identify
two isogenies if they have the same kernel.
% We will consider
%two isogenies $\phi, \psi$ equivalent if $\phi  = \alpha \circ \psi \circ
%\alpha'$ for isomorphisms $\alpha, \alpha'$. 
%This means that we can
%describe a separable isogeny of $E$ to some other elliptic curve by giving its
%kernel. 
We can compute equations for the isogeny from its kernel by using
V\'elu's formula~\cite{Vel71}, \S\ref{sec:velusFormula}.

An isogeny of an elliptic curve $E$ to itself is called an
endomorphism of $E$. If $E$ is defined over some finite field
$\mathbb{F}_q$, then the set of endomorphisms of $E$
defined over $\Fqbar$ together with the zero map form
a ring under the operations addition and composition. It is called the
\textit{endomorphism ring} of $E$, and is denoted by $\End(E)$. It is isomorphic either to an
order in a quadratic imaginary field or to an order in a quaternion
algebra. In the first case we call $E$ an {\em ordinary elliptic
  curve}. An elliptic curve whose endomorphism is isomorphic to an
order in a quaternion algebra is called a {\em supersingular elliptic
  curve}.  Every supersingular elliptic curve over a field of
characteristic $p$ has a model that is defined over $\mathbb{F}_{p^2}$
because the $j$-invariant of such a curve is in
$\mathbb{F}_{p^2}$. Given $j\in \Fqbar$ such that
$j\neq0,1728$, we write $E(j)$ for the curve defined by
the equation 
\begin{equation}\label{ellcurve}
y^2+xy=x^3-\frac{36}{j-1728}x-\frac{1}{j-1728}.
\end{equation}
Such a curve can be put into a short Weierstrass equation
$y^2=x^3+Ax+B$.  
We also
write $E(0)$ and $E(1728)$ for the curves with equations $y^2=x^3+1$
and $y^2=x^3+x$, respectively.

We borrow definitions from~\cite[2.3]{Sut13} to define the $\ell$-isogeny
graph of supersingular elliptic curves in characteristic $p$. 
Given a prime $\ell$, the {\em modular polynomial} $\Phi_{\ell}(X,Y)\in
\mathbb{Z}[X,Y]$ has the property that if $j,j'\in \mathbb{F}_q$ then
$\Phi_{\ell}(j,j')=0$ if and only if there are elliptic curves
$E,E'/\mathbb{F}_q$ with $j$-invariants $j$ and $j'$,
respectively, such that there is a separable isogeny $\phi: E\to E'$ of degree
$\ell$. When $j\not=0,1728$, $j,j'\in \F_q$, and $E$ is ordinary, we can choose $E'$ and
$\phi$ such that
$\phi$ is defined over $\mathbb{F}_q$~\cite[Proposition
6.1]{Schoof95}. We state and prove a similar result for supersingular elliptic
curves $E$ in Corollary~\ref{cor:isogenydefined}. 

\begin{definition}\label{deg: isogeny graph}
%	Let $\F_p$ be a finite field of characteristic $p$, and l
	Let $\ell$ be a prime different from $p$. The
        \textit{supersingular $\ell$-isogeny graph in characteristic $p$} is
        the multi-graph $G(p,\ell)$ whose vertex set is 
\[
V=V(G(p,\ell)) = \{ j \in \mathbb{F}_{p^2}: E(j) \text{ is
  supersingular}\},
\]
and the number of directed edges from $j$ to $j'$ is equal to
multiplicity of $j'$ as a root of $\Phi_{\ell}(j,Y)$. We can identify
an edge between $E(j)$ and $E(j')$ with a cyclic isogeny
$\phi: E(j)\to E(j')$ of degree $\ell$. Such an isogeny is unique up
to post-composing with an automorphism of $E(j')$. By the above
discussion, we can label the edges of $G(p,\ell)$ which start at a
given vertex $E(j)$ with the $\ell+1$ chosen isogenies whose kernels
are the nontrivial cyclic subgroups of
$E(j)[\ell]$.  
\end{definition}

%\begin{remark}
%\travis{As observed in~\cite{Kohel}, every edge has a dual, but it can be the
%case that an edge is dual to two or three edges, and not every edge
%appears as the dual of another edge. }
%\end{remark}

%A = A(G(p,\ell)) = \{\phi_K: E(j) \to E(j'):
%  \textup{$\phi_K$ isogeny}, K = \ker(\phi_K) \subseteq
 %   E(j)[\ell]$ cyclic subgroup$\}$;
%\item $s:A\to V$ defined by $s(\phi:E(j) \to E(j')) = E(j)$;
%\item $t:A\to V$ defined by $t(\phi: E(j)\to E(j')) = E(j')$. 
%\end{itemize}
%\end{definition}

%For any prime $\ell\neq p$, one can construct a so-called
%Now the \emph{$\ell$-isogeny graph} is the graph in which each vertex
%is $E(j)$ where $j$ is a
%supersingular $j$-invariant, and an edge between two vertices $E(j_1)$
%to $E(j_2)$ is
%associated to a degree $\ell$ isogeny between the corresponding
%curves $E_{j_1}$ and $E_{j_2}$. 

%Let $\mbox{Super}_{p}$ denote a complete set of representatives of
%isomorphism classes of supersingular elliptic curves. 
%Then
%	\begin{itemize}
%		\item $V\Def \{j \in \mathbb{F}_{p^2} : j=j([E]) \mbox{ for some } [E]\in\mbox{Super}_{p} \}.$
%		\item $E\Def \{ (j,j')\in V\times V : \mbox{ There exists an } \ell\mbox{-isogeny } [E_j]\longrightarrow [E_{j'}] \}$
%	\end{itemize}

%\begin{remark} 
%	Note that $G$ is a \textit{directed multigraph}, as we allow multiple edges between vertices, as well as self-loops.
%\end{remark}

Let $E,E'$ be two supersingular elliptic curves defined over
$\mathbb{F}_{p^2}$. For each prime $\ell \neq p$, $E$ and
$E'$ are connected by a chain of isogenies of degree
$\ell$~\cite{Mes86}. By~\cite[Theorem79]{Kohel}, $E$ and $E'$ can be
connected by $m$ isogenies of degree $\ell$ (and hence by a single isogeny of degree $\ell^m$), where $m=O(\log p)$. 
%The diameter of the graph $G(p,\ell)$ is $O(\log p)$~\cite[Theorem79]{Kohel}({\tt reference??}).
If $\ell$ is a fixed prime such that $\ell=O(\log p)$, then any $\ell$-isogeny in the
chain above can either be specified by rational maps or by giving the
kernel of the isogeny, and both of these representations have
size 
polynomial in $\log p$. %By V\'elu's formula, and since $\ell =O(\log p)$, there
%is an efficient way to go back and forth between the representation of
%an isogeny by its kernel or by rational maps.

The theorem below summarizes several properties of the supersingular $\ell$-isogeny
graph mentioned above.
\begin{theorem}\label{theorem: simple properties}\normalfont
	Let $\ell \neq p$ be a prime, and let $G(p,\ell)$ be the
        supersingular $\ell$-isogeny graph in characteristic $p$ as in Definition~\ref{deg: isogeny graph}.
	\begin{enumerate}
		\item $G$ is connected.
		\item $G$ is $(\ell+1)$-regular as a directed graph.
		\item $\# V = \left[\frac{p}{12}\right]+\varepsilon_p$. Here, 
		\begin{equation*}\label{eq:num of vertices}
			\varepsilon_p \Def \begin{cases}
				0, & p\equiv 1 \pmod{12}\\
				1, & p=3\\
				1, & p\equiv 5,7 \pmod{12}\\
				2, & p\equiv 11\pmod{12}
			\end{cases}
		\end{equation*}
	\end{enumerate}
\end{theorem}

\begin{remark}
  We have an exception to the $\ell+1$-regularity at the vertices and
  their neighbors corresponding to elliptic curves with $j = 0, 1728$,
  due to their extra automorphisms.
\end{remark}

As $G(p,\ell)$ is connected, for any two supersingular elliptic curves
$E,E'$ defined over $\F_{p^2}$ there is an isogeny $\phi:E \to E'$ of
degree $\ell^e$ for some $e$. In fact, we can take this isogeny to be
defined over an extension of $\F_{p^2}$ of degree at most 6. If
$E/\F_q$ is an elliptic curve and $n\geq 1$ is an integer, we denote
the ring of endomorphisms of $E$ which are defined over $\F_{q^n}$ by
$\End_{\F_{q^n}}(E)$. 
%The proposition and its proof below were
%communicated by Andrew Sutherland.

\begin{proposition}\label{prop:endsdefined}
Let $E/\F_{p^2}$ be a supersingular elliptic curve. Then
$\End_{\F_{p^{2d}}}(E)=\End(E)$, where $d=1$ if $j(E)\neq 0,1728$,
  $d=1$ or $d=3$ if $j(E)=0$, and $d=1$ or $d=2$ if $j(E)=1728$. 
\end{proposition}

\begin{proof}
  First, we observe that for a supersingular elliptic curve
  $E/\F_{q}$, all endomorphisms of $E$ are defined over $\F_q$ if and
  only if the Frobenius endomorphism $\pi: E\to E$ is equal to $[p^k]$
  or $[-p^k]$ 
  for some $k$. This follows from Theorem 4.1 of~\cite{waterhouse},
  case (2). This is the case when $q$ is an even power of $p$ and the
  trace of Frobenius is $\pm 2\sqrt{q}$.

  Now assume $E/\F_{p^2}$ is a supersingular elliptic curve and let
  $\pi: E\to E$ be the Frobenius endomorphism of $E$. Consider the
  multiplication-by-$p$ map $[p]: E\to E$. By~\cite[II.2.12]{AEC}, the
  map $[p]$ factors as
\[
[p] = \alpha \circ \pi,
\]
where $\alpha$ is an automorphism of $E$. The automorphism is defined
over $\F_{p^2}$, since $[p]$ and $\pi$ are defined over
$\F_{p^2}$. Thus $\pi$ and $\alpha$ commute. If $j(E)\not=0,1728$,
then $\Aut(E)=\{[\pm 1]\}$ and thus $[p]=\pm \pi$. By the above observation,
we have $\End_{F_{p^2}}(E) = \End(E)$.

If $j(E)=0$, then $\alpha^3 = [\pm 1]$. We also have
\[
[p^3] = (\alpha\circ \pi)^3 = \alpha^3 \circ \pi^3,
\]
so $\pi^3 = [\pm p^3]$. In this case, the Frobenius of the base change
of $E$ to $\F_{p^6}$ is $\pi^3$, and thus
$\End_{\F_{p^6}}(E) = \End(E)$ again by the above observation. The proof of
the case $j(E)=1728$ is similar.
\end{proof}

\begin{corollary}\label{cor:isogenydefined}
  If $E_0,E_1/\F_{p^2}$ are supersingular and if $\ell\not= p$, then
  $\Hom_{\F_{p^{2d}}}(E_0, E_1) = \Hom(E_0,E_1)$, where 
  $d=1,2,3,$ or $6$. If $j(E_i)\not=0,1728$ for $i=0,1$, then we can
  take $d=2$. 
\end{corollary}
\begin{proof}
  %There is an isogeny $\phi: E_0\to E_1$ of degree $\ell^e$. We now wish to
  %show that $\Hom_{\F_{p^{2d}}}(E_0,E_1) = \Hom(E_0,E_1)$, and thus
  %$\phi$ is defined over $\F_{p^{2d}}$, for $d=1,2,3,$ or $6$. 
  
  First, we claim that for
  some $d'\in\{1,2,3,6\}$, we have
  $\#E_0(\F_{p^{2d'}})=\#E_1(\F_{p^{2d'}})$. This follows from the
  fact that all supersingular curves $E/\F_{p^2}$ will have the same
  number of points over an extension of $\F_{p^2}$ of degree $1,2,3,$
  or $6$, which we now prove. The trace of Frobenius of
  a supersingular elliptic curve $E/\F_{p^2}$ is either $0, \pm p$ or
  $\pm 2p$, again by Proposition 4.1 of~\cite{waterhouse}. By
  inspection, we see that the roots $\alpha,\beta$ of the
  characteristic polynomial of the ($p^2$-power) Frobenius
  endomorphism of $E$, then $\alpha$ and $\beta$ are either $\pm p$,
  $\pm p(1\pm \sqrt{-3})/2$, $\alpha=\beta=p$, or
  $\alpha=\beta=-p$. Thus for some $d'\in\{1,2,3,6\}$,
  $\alpha^{d'}=\beta^{d'}=p^{d'}$. By~\cite[V.2.3.1(a)]{AEC},
  \[
\#E(\F_{p^{2d'}}) = p^{2d'}+1 - \alpha^{d'} - \beta^{d'} = p^{2d'}+1-2p^{d'}.
\]
Thus we can choose $d'=1,2,3,$ or $6$ such that the claim holds for
$E_0$ and $E_1$. 

  Now let
  $\F_{p^{2d''}}/\F_{p^2}$ be an extension such that
  $\End_{\F_{p^{2d''}}}(E_i)=\End(E_i)$ for $i=0,1$. Let
  $d=\lcm\{d',d''\}$.  By Proposition~\ref{prop:endsdefined}, $d\in
  \{1,2,3,6\}$. 

%By Proposition~\ref{prop:endsdefined}, we have
 % $\End_{\F_{p^2}}(E_i)=\End(E_i)$. Either $E_0$ or its quadratic
  %twist has the same number of $\F_{p^2}$-rational points as $E_1$,
  %and thus they have the same number of points over either $\F_{p^2}$
  %or $\F_{p^4}$. If $j(E_0)=1728$, then $p\equiv 3 \pmod{4}$ and we
  %are in either case (2) or (5) $d$ is either $1,2,$ or $4$ by
  %Proposition~\ref{prop:endsdefined} and because $E_0$ has quartic
  %twists which it becomes isomorphic to over $\F_{p^{8}}$~\cite{

  If $d>1$, we base change $E_i$ to $\F_{p^{2d}}$ and denote these
  curves again by $E_i$. For each $i=0,1$, let $\pi_{E_i}$ be the
  $p^{2d}$-power Frobenius endomorphism for $E_i$; then by our choice
  of $d'$ which divides $d$, either $\pi_{E_i}=[p^d]: E_i\to E_i$ or
  $\pi_{E_i} = [-p^d]: E_i\to E_i$ for
  $i=0,1$. Consequently, for any $\phi\in \Hom(E_0,E_1)$ it follows
  that
\[
\phi \circ \pi_{E_0} = \pi_{E_1} \circ \phi.
\]
Thus $\phi$ is defined over $\F_{p^{2d}}$. 
\end{proof}

\begin{remark}
If $E_0,E_1/\F_{p^2}$ are supersingular and $j(E_i)\not=0,1728$ for
$i=0,1$, then $E_0$ and $E_1$ are connected by a chain of
$\ell$-isogenies defined over $\F_{p^4}$. 
\end{remark}

%\jenn{more definitions and properties? We could make references to the work of Pizer (presumably saying that $G$ is Ramanujan, but we may not need this in this paper)}
%\efrat{ I don't know what else should be in this section. maybe a very small example? not sure yet. Discuss that in view of the facts that there are "many" vertices, "many" edges" and that the graph is $\ell+1$-regular, leads one to try to find some " relaxation " before finding *all* loops and more. In other words, justify why the obstructions we find could potentially save time when searching through the graph. It will be nice if we could estimate how much time we can save by ignoring the known obstructions.}
%\efrat{ maybe discuss here the meaning of multiple edges, in the sense that a map can be dual to itself, or not, and how can this be detected from the graph. }

%%%%%%%%%%%%%%%%%%%%%%%%%%%%%%%%%%%%%%%%%%%%%%%%%%%%%%%%%%%%%%%%%%%%%%%%%%%%%%%%%%%%%%%%%%%%%%%%%%%%%%%%%%%%%
%%%%%%%%%%%%%%%%%%%%%%%%%%%%%%%%%%%%%%%%%%%%%%%%%%%%%%%%%%%%%%%%%%%%%%%%%%%%%%%%%%%%%%%%%%%%%%%%%%%%%%%%%%%%%
\section{Quaternion algebras, endomorphism rings, and
  cycles}\label{sec: endo rings computation}

%%%%%%%%%%%%%%%%%%%%%%%%%%%%%%%%%%%%%%%%%%%%%%%%%%%%%%%%%%%%%%%%%%%%%%%%%%%%%%%%%%%%%%%%%%%%%%%%%%%%%%%%%%%%%
\subsection{Quaternion algebras}\label{sec: theoretical aspects}

For $a, b \in \QQ^{\times}$, let $H(a,b)$ denote the quaternion
algebra over $\QQ$ with basis $1,i,j,ij$ such that $i^2=a$, $j^2=b$
and $ij=-ji$.  That is,
\[
H(a,b) = \QQ + \QQ \,i + \QQ \,j  + \QQ \,i j.
\]
Every 4-dimensional central simple algebra over $\mathbb{Q}$ is isomorphic to $H(a, b)$
for some $a, b \in \QQ$; for example, see~\cite[Proposition 7.6.1]{Voight}. %\jenn{cite?}

There is a {\em canonical involution} on $H(a,b)$ which sends an
element $\alpha~=~a_1+a_2i+a_3j+a_4ij$ to $\overline{\alpha}:=a_1-a_2i-a_3j-a_4ij$.
Define the {\em reduced trace} of an element $\alpha$ as above to
be \[\Trd(\alpha) = \alpha + \overline{\alpha}= 2a_1,\] and the {\em reduced
norm} to be 
\[
\Nrd(\alpha) = \alpha \overline{\alpha}= a_1^2 - aa_2^2 -ba_3^2 + aba_4^2.
\]

\begin{definition}
Let $B$ be a quaternion algebra over $\QQ$, and let $p$ be a prime or
$\infty$. Let $\QQ_p$ be the $p$-adic rationals if $p$ is finite, and
let $\QQ_{\infty} =\mathbb{R}$. 
 We say that $B$ is {\em split at $p$}  if 
\[
B \otimes_{\QQ} \QQ_p \cong M_2(\QQ_p),
\]
where $M_2(K)$ is the algebra of $2 \times 2$ matrices with
coefficients in $K$. Otherwise $B$ is said to be {\em ramified at $p$}.
\end{definition}

%In \cite{Deuring} Deuring proved the following theorem.
Orders in quaternion algebras appear as endomorphism rings of some elliptic curves (\cite{Deuring}):
\begin{theorem}[Deuring Correspondence]\label{thm:deuringcorrespondence}
  Let $E$ be an elliptic curve over $\F_p$ and suppose that the
  $\Z$-rank of $\End(E)$ is $4$. Then $B:=\End(E) \otimes_\Z \Q$ is a
  quaternion algebra ramified exactly at $p$ and $\infty$, denoted $B_{p, \infty}$, and
  $\End(E)$ is isomorphic to a
  maximal order in $B_{p,\infty}$.
\end{theorem}

Under this isomorphism, taking the dual isogeny on $\End(E)$ corresponds to the canonical
involution in the quaternion algebra, and thus the degrees and traces of
endomorphisms correspond to reduced norms and reduced traces of elements in the
quaternion algebra. 

%We denote the quaternion algebra ramified exactly at $p$ and $\infty$
%by $B_{p, \infty}.$

\begin{lemma}[\cite{Piz80}, Proposition 5.1]\label{lem:b_p_infty} $B_{p, \infty}$ can be explicitly given
  as %\jenn{list of quaternion algebras}
	\begin{enumerate}
		\item[$(i)$]$B_{p,\infty}=\left(\frac{-1,-1}{\Q}\right)$ if $p=2$;
		\item[$(ii)$]$B_{p,\infty}=\left(\frac{-1,p}{\Q}\right)$ if $p\equiv 3 \pmod 4$;
		\item[$(iii)$]$B_{p,\infty}=\left(\frac{-2,-p}{\Q}\right)$ if $p\equiv 5 \pmod 8$ and 
		\item[$(iv)$]$B_{p,\infty}=\left(\frac{-p,-q}{\Q}\right)$
                  if $p\equiv 1 \mod 8$, where $q\equiv 3 \mod 4$ is a
                  prime such that $q$ is not a square modulo $p$.
	\end{enumerate}
\end{lemma}

The quaternion algebra
$B_{p,\infty}$ is an inner product space with respect to the bilinear
form
\[
\langle x,y \rangle =\frac{\Nrd(x+y)- \Nrd(x)-\Nrd(y)}{2}.
\]
The basis $\{1,i,j,ij\}$ is an orthogonal basis with respect to this
inner product.

%\efrat{explain the difficulties in finding the actual
  %correspondence. Then, suggest how to find this, using the remarks of
  %Jenn+ kirsten below. isogenies vs elements of the ring vs loops in
  %the graph.}

%\jenn{Talk about the fact that loops can be realized as elements of
 % quaternion algebras. Kohel suggests that studying loops in a single
 % $\ell$-isogeny graph suffices to generate a maximal order.}

%We can give $B_{p,\infty}$ the structure of a quadratic space, where the norm is the
%reduced norm of $B_{p,\infty}$. The corresponding inner product is 
%\[
%\frac{1}{2}\Trd(x\overline{y}) =
%\frac{1}{2}(\Nrd(x+y)-\Nrd(x)-\Nrd(y)).
%\]

%This above proposition tells us that it is thus sufficient to specify an endomorphism ring of a supersingular elliptic curve by the% traces and norms of its two generators.
%%%%%%%%%%%%%%%%%%%%%%%%%%%%%%%%%%%%%%%%%%%%%%%%%%%%%%%%%%%%%%%%%%%%%%%%%%%%%%%%%%%%%%%%%%%%%%%%%%%%%%%%%%%%%
\subsection{Computing endomorphism rings from cycles in the
  $\ell$-isogeny graphs}\label{sec: computational aspects}

%The quaternion algebra $B_{p,\infty}$ has many different non-isomorphic
%maximal orders. One method to compute the endomorphism ring of a
%supersingular elliptic curve $E$ as a maximal order in $B_{p,\infty}$,
%as carried out in~\ref{Cer14,LM04}, is to match up all possible
%maximal orders with supersingular $j$-invariants.

Suppose we have an order in $\End(E)$ generated by two cycles in
$G(p,\ell)$, which we embed as $\mathcal{O}\subseteq
B_{p,\infty}$. Suppose we also have another order
$\mathcal{O}'\subseteq B_{p,\infty}$. We want to check whether
$\mathcal{O}\simeq\mathcal{O}'$ or not. 

One can check this via using the fact that two orders are isomorphic
if and only if they are conjugate; this follows from the
Skolem-Noether theorem, see~\cite[Lemma 17.7.2]{Voight}, for
example. One can do this by showing that as lattices in a quadratic
space $\End(E)\otimes \QQ \isom B_{p,\infty}$, the two lattices are
isometric under the quadratic form induced by $\Nrd$. Thus, we can
check whether $\mathcal{O}\simeq \mathcal{O}'$ by computing Gram
matrices for a basis of each, and checking whether the matrices are
conjugate by an orthogonal matrix. The following proposition, which is
Corollary~4.4 in~\cite{Neb98}, makes this remark explicit.

\begin{proposition}
 Two orders $\OO,\OO'\subseteq B_{p,\infty}$ are conjugate if and only
 if they are isometric as lattices with respect to the inner product
 induced by $\Nrd$. In particular, for
  $m,n\in \{1,\ldots,4\}$, let $x_m,y_n$ be elements in the
  quaternion algebra $B_{p,\infty}$ such that
  $\mathcal{O}_1= \langle x_1,x_2,x_3,x_4\rangle$ and
  $\mathcal{O}_2= \langle y_1,y_2,y_3,y_4\rangle$ are orders in
  $B_{p,\infty}$.
  If $\Trd(x_m \overline{x_n})=\Trd(y_m\overline{y_n})$
  for $m,n \in \{1,2,3,4\}$, then $\mathcal{O}_1 \cong \mathcal{O}_2$.
\end{proposition}

\begin{proof}
The first statement is~\cite[Corollary 4.4]{Neb98}. The second
statement follows then from the first:
the map $x_m\to y_n$ 
extends linearly to an isometry of lattices in
$B_{p,\infty}$. This implies that
$\mathcal{O}_1$ and $\mathcal{O}_2$ are conjugate in $B_{p,\infty}$ and hence isomorphic as orders. 
\end{proof}

Thus if we have two cycles in $G(p,\ell)$ passing through $E(j)$ which
correspond to endomorphisms $\alpha,\beta\in \End(E(j))$, we can generate
an order 
\[
\OO=\langle 1,\alpha,\beta,\alpha\beta\rangle=\langle
x_0,x_1,x_2,x_3 \rangle \subseteq \End(E) \otimes_{\Z} \Q.
\]
 Also, suppose we have an order $\OO'=\langle
y_0,y_1,y_2,y_3\rangle\subseteq B_{p,\infty}$. Then we can check whether
$\OO\simeq \OO'$ by comparing $\tr(x_m\widehat{x_n})$ and
$\Trd(y_m\overline{y_n})$. This idea is used in our examples in
\S\ref{sec: examples}. Additionally, we use this idea in
Theorem~\ref{thm:nonmaximal order} to produce a geometric
obstruction to generating $\End(E(j))$ by two cycles in $G(p,\ell)$. 
%\begin{lemma}\label{lemma:endofromloop}
%Let $E/\overline{\mathbb{F}_p}$ be a supersingular elliptic curve. Let
%$\phi:E\to E'$ be a primitive isogeny satisfying
%$\deg(\phi)=\ell^e$ for some $e\in \N$. Let $j=j(E)$. Then there is
%a unique path in $G(p,\ell)$ corresponding to 
%\end{lemma}

\begin{lemma}
  Let $\{a_1,\ldots,a_e\}$ be a cycle beginning and ending at a vertex
  $E(j)$. Then the endomorphism of $E(j)$ corresponding to this cycle
  has degree $\ell^e$.
\end{lemma}
\begin{proof}
  Each edge $a_k$ represents an $\ell$-isogeny, which has degree
  $\ell$. Composition of $N$ isogenies of degree $\ell$ results in an
  isogeny of degree $\ell^N$. 
\end{proof}

\begin{theorem}
  Let $C=\{a_1,\ldots,a_e\}$ be a cycle in $G(p,\ell)$ beginning and
  ending at a vertex $E(j)$ corresponding to an endomorphism of
  $E(j)$.  Then the (reduced) trace of $C$ interpreted as an element
  of $\End(E(j))$ can be computed in time polynomial in $\ell$,
  $\log p$, and $e$.
\end{theorem}
\begin{proof}
This is proved in the appendix.
\end{proof}
In fact, some of the traces can be recognized immediately without resorting to the modification of Schoof's algorithm.

\begin{lemma}\label{lemma:traceofconjugate}
  The cycles corresponding to the multiplication-by-$\ell^n$ map ($n$ of
  the $\ell$-isogenies followed by their dual isogenies in reverse
  order) have trace $2\ell^n$. Suppose $\phi:E\to E'$ is an isogeny
  and $\rho\in \End(E')$. Then
  $\tr(\hat\phi \circ \rho \circ \phi) = \deg(\phi) \cdot \tr(\rho)$.
\end{lemma}
\begin{proof}
Let $\phi: E\to E'$ be an isogeny of supersingular elliptic curves. By
Proposition 3.9 of \cite{waterhouse}, the map 
\begin{align*}
\iota\otimes \text{id}: \End(E')\otimes \QQ &\to \End(E)\otimes \QQ \\
\rho\otimes 1 &\mapsto \widehat{\phi}\rho\phi\otimes\frac{1}{\deg(\phi)}
\end{align*}
is an isomorphism of quaternion algebras. It follows that
$\tr(\widehat{\phi}\rho\phi)=\deg(\phi) \tr(\rho)$.
 \end{proof}
%%%%%%%%%%%%%%%%%%%%%%%%%%%%%%%%%%%%%%%%%%%%%%%%%%%%%%%%%%%%%%%%%%%%%%%%%%%%%%%%%%%%%%%%%%%%%%%%%%%%%%%%%%%%%
\section{A condition for linear independence}\label{sec: sufficient conditions}
In this section, we prove a necessary and sufficient condition for two
endomorphisms $\alpha$ and $\beta$ to be linearly independent. To
prove this we need the notion of a {\em cycle which has no
  backtracking}. We first show that this notion is equivalent to the
corresponding endomorphism being primitive. Then, in
Theorem~\ref{thm:independent endos}, we characterize when two cycles
with no backtracking are linearly independent. To do this we use the
fact that if two endomorphisms are linearly dependent, then they
generate a subring of a quadratic imaginary field, and in particular,
they must commute. As a corollary, we obtain that two cycles through a
vertex $E(j)$ that do not have the same vertex set must be linearly independent.

\begin{definition}
An isogeny $\phi:E \to E'$ is {\em primitive} if it does not
factor through $[n]: E\to E$ for any natural number
$n>1$. 
\end{definition}
\begin{remark}
An isogeny $\phi:E\to E'$ is primitive if $\ker(\phi)$ does
not contain $E[n]$ for any $n>1$. 
\end{remark}
 \begin{definition}
Suppose $a_1,a_2$ are edges in $G(p,\ell)$ whose chosen
representatives are $\ell$-isogenies
$\phi: E(j) \to E(j')$, $\psi: E(j')\to E(j)$. We say that $a_2$ is
dual to $a_1$ if
$\hat{\phi} \in \Aut(E(j))\psi$. 
A cycle $\{a_1,\ldots,a_e\}$ in
$G(p,\ell)$ {\em has no backtracking} if $a_{i+1}$ is not dual to
$a_i$ for $i=1,\ldots,e-1$. 
\end{definition}
\begin{remark}
  Let $\{a_1,\ldots,a_e\}$ be a path in $G(p,\ell)$ and let
  $\phi_i: E(j_i)\to E(j_{i+1})$ be the chosen isogeny representing $a_i$ for
  $i=1,\ldots,e$. Suppose that $a_{k+1}$ is dual to $a_k$ for some
  $1\leq k\leq e-1$. Then we claim that the isogeny
\[
\phi = \phi_e\circ\cdots\circ\phi_1
\]
will not be primitive. Since $a_{k+1}$ is dual to $a_k$, there
exists $\rho\in \Aut(E(j_k))$ such that
$\phi_{k} = \widehat{\phi_{k+1}}\rho$. Then
$\phi_{k+1}\circ \phi_k = [\ell] \rho$, so $\phi$ factors through
$[\ell]$.
\end{remark}

Our definition of a cycle with no backtracking is less restrictive
than the notion of a simple cycle in \cite{Kohel}, which additionally
requires that there are no repeated vertices in the
cycle. Proposition~82 of~\cite{Kohel} shows that simple cycles in
$G(p,\ell)$ through $E(j)$ give rise to primitive endomorphisms. We
strengthen this result, proving in Lemma~\ref{lemma:nobacktracking}
below that cycles through $E(j)$ with no backtracking correspond
exactly to primitive endomorphisms.

Given a path in $G(p,\ell)$ of length $e$ between $j$ and $j'$,
there is an isogeny $\phi: E(j)\to E(j')$ of degree $\ell^e$ obtained
by composing isogenies representing the edges in the path. If this
path has no backtracking, the kernel of $\phi$ is a cyclic subgroup of
order $\ell^e$ in $E(j)[\ell^e]$. Conversely, given an
isogeny $\phi: E(j)\to E(j')$ with cyclic kernel of order $\ell^e$, 
there is a corresponding path in $G(p,\ell)$.
\begin{proposition}\label{prop: isog to path}
Suppose that $\phi: E(j)\to E(j')$ is an isogeny with cyclic kernel of
order $\ell^e$. There is a unique path in $G(p,\ell)$ such that the
factorization of $\phi$ into a chain of $\ell$-isogenies corresponds
to the edges in the path, and the path has no backtracking. 
\end{proposition}
\begin{proof}
  The proof is by induction on $e$. If $e=1$, there is a unique edge
  corresponding to the isogeny $\phi:E(j)\to E(j')$, because each edge
  starting at $E(j)$ corresponds to a unique cyclic subgroup of
  $E(j)[\ell]$. Now suppose that the kernel of $\phi:E(j)\to E(j')$ is
  generated by a point $P$ of $E(j)$ of order $\ell^{e}$. There is an
  edge in $G(p,\ell)$ from $E(j)$ to another vertex $E(j_1)$ which is
  labeled by $\phi_1: E(j)\to E(j_1)$ and whose kernel is
  $\langle [\ell^{e-1}]P\rangle$. Then because $\phi([\ell^{e-1}]P)=0$, we
  have a factorization $\phi:=\psi\circ \phi_1$. Then
  $\psi: E(j_1)\to E(j')$ has degree $\ell^{e-1}$ and its kernel is cyclic
  of order $\ell^{e-1}$, generated by $\phi_1(P)$. Then there is a path of
  length $e-1$ between $E(j_1)$ and $E(j')$ with no backtracking by the
  inductive hypothesis. By concatenating with the edge corresponding
  to $\phi_1$, we have a path of length $e$ between $E(j)$ and
  $E(j')$. Note that the first edge in the path for $\psi$ can not be
  dual to the edge for $\phi_1$, because otherwise
  $E(j)[\ell]\subseteq \ker\phi$, which is cyclic by assumption.
\end{proof}

Given a path $C$ in $G(p,\ell)$ starting at $E(j)$, the {\em isogeny
  corresponding to $C$} is the isogeny obtained by composing the
isogenies represented the edges along the path. Conversely, given an
isogeny $\phi:E(j)\to E(j')$ with cyclic kernel, the {\em path
  corresponding to $\phi$} is the path constructed as
above. We remark that it is the kernel of an isogeny which
  determines the path in $G(p,\ell)$, so two distinct primitive
  isogenies will determine the same path if they have the same
  kernel. This path is only unique because we fix an isogeny
  representing each edge.
  %in
  %$G(p,\ell)$ is only unique because we fix an isogeny for each edge in
  %$G(p,\ell)$. This is because if $\phi: E(j) \to E(j')$ has kernel
  %$\langle P \rangle$ of order $\ell^e$ and $\phi$ and factors through
  %an isogeny $\psi: E(j) \to E'$ of degree $k<e$, where
  %$j(E') = 0$ or $1728$, and if $\rho\in \Aut(E')$ with $\rho\not=
  %[\pm 1]$, then
  %$\psi([\ell^{e-k-1}]P)$ and $\rho(\psi([\ell^{e-k-1}]P))$ could generate
  %different subgroups of $E[\ell]$. But if an edge is labeled by a
  %specific isogeny rather than a subgroup, there is no ambiguity,
  %because we factor $\phi$ and only use our chosen edges.}

%\kirsten{We think this argument should also work if $\alpha$ is not
  %the shortest loop.}
\begin{lemma}\label{lemma:nobacktracking}
Let $\{a_1,\ldots,a_e\}$ be a cycle in $G(p,\ell)$ through the vertex $E(j)$ with
corresponding endomorphism $\alpha\in \End(E(j))$. If the cycle has no backtracking, then 
the corresponding endomorphism $\alpha\in \End(E(j))$ is
primitive. Conversely, if $\alpha\in \End(E(j))$ is primitive and
$\deg(\alpha)=\ell^e$ for some $e\in \mathbb{N}$, the cycle in
$G(p,\ell)$ corresponding to $\alpha$ has no backtracking. 
\end{lemma}
\begin{proof}
The first statement is proved as Proposition~82 in~\cite{Kohel}. His
proof does not use the assumption that there are no repeated vertices
in the cycle. Now
assume that $\alpha\in \End(E(j))$ is primitive and
$\deg(\alpha)=\ell^e$. Then by Proposition~10 of~\cite{EHLMP}, the
kernel of $\alpha$ is cyclic, generated by
$P\in E[\ell^e]$. By Proposition~\ref{prop: isog to path}, the cycle
in $G(p,\ell)$ corresponding to $\alpha$ has no backtracking. 

%We claim that $\alpha$ uniquely determines a
%path in $G(p,\ell)$.  We can decompose $\alpha$ as
%$\alpha=\phi_e\circ\cdots\circ\phi_1$, where $\phi_i:E_i\to E_{i+1}$
%has degree $\ell$. This
%factorization is unique up to post-composition of $\phi_i$ with an
%isomorphism by~\cite[III.4.12]{AEC}. No isogeny $\phi_i$ is followed by its
%dual, since 
%$\alpha$ does not factor through $[\ell]:E\to E$. This factorization
%yields the desired cycle with no backtracking in $G(p,\ell)$. 
\end{proof}

Suppose $\alpha\in \End(E(j))$ is an endomorphism of degree
$\ell^e$. We wish to describe what information we can infer about the
order $\Z[\alpha]$ of $\QQ(\alpha)$ from the cycle corresponding to
$\alpha$ in $G(p,\ell)$. We will show that we can detect when
$\Z[\alpha]$ is maximal at a prime above $\ell$. 

\begin{lemma}\label{lemma: nonzero trace}
  Let $\alpha \in \End(E(j))$ be a primitive endomorphism
  corresponding to a cycle $\{a_1,\ldots,a_e\}$ in $G(p,\ell)$ which
  begins at $E(j)$. Then $a_1$ is dual to $a_e$ if and only if $\tr(\alpha) \equiv
  0 \pmod{\ell}$. 
\end{lemma}

\begin{proof}
  The endomorphism $\alpha$ determines an endomorphism
  $A = \alpha|_{E[\ell]}$ of $E[\ell]$. If
  $\tr(\alpha)\equiv 0 \pmod{\ell}$, then the characteristic
  polynomial of $A$ is $x^2$. Thus $E[\ell]\subseteq \ker(\alpha^2)$,
  so $\alpha^2$ is not primitive. Lemma~\ref{lemma:nobacktracking}
  implies that the cycle $\{a_1,\ldots,a_e,a_1,\ldots,a_e\}$ in
  $G(p,\ell)$ has backtracking, because the endomorphism
  corresponding to this cycle is $\alpha^2$. We must have that $a_1$ is dual to $a_e$ because $\alpha$ has no
  backtracking. 

  Conversely, assume $a_1$ is dual to $a_e$. Suppose that $a_e$ is an
  edge from the vertex $E(j')$, and let $\phi_1: E(j)\to E(j')$ be the
  isogeny corresponding to $a_1$ and let $\phi_e: E(j')\to E(j)$ be
  the isogeny corresponding to $a_e$. Then $\phi_e = \hat{\phi_1}u$
  for some $u\in \Aut(E(j))$. Thus
  $\alpha= \widehat{\phi_1} \alpha'\phi_1$ where $\alpha'$ is an
  endomorphism of $E(j')$. By
  Proposition~\ref{lemma:traceofconjugate},
  $\tr(\alpha)\equiv 0 \pmod{\ell}$.
\end{proof}

%\begin{remark}The previous lemma implies that as long as we work with
%a cycle whose first edge is not dual to its last edge, the
%corresponding endomorphism will generate an order whose conductor is
%coprime to $\ell$. 
%\end{remark}
This lets us conclude the following. 

\begin{lemma}\label{lemma: ell-maximal}
  Let $\{a_1,\ldots,a_e\}$ be a cycle in $G(p,\ell)$ with no
  backtracking and such that $a_1$ is not dual to $a_e$. Suppose $a_1$
  is an edge originating from $E(j)$. In the case that the cycle is a
  self-loop $a_1$ at $E(j)$, we assume that $a_1$ is not dual to
  itself.  Let $\alpha\in \End(E(j))$ be the endomorphism
  corresponding to the cycle. Then the conductor of the quadratic
  order $\Z[\alpha]$ in $\QQ(\alpha)$ is coprime to $\ell$.
\end{lemma}
\begin{proof}
  As $\alpha$ is primitive, it determines a quadratic imaginary
  extension $\QQ(\alpha)$ of $\QQ$. The discriminant of $\alpha$ is
  $\tr(\alpha)^2-4\ell^e$, which is coprime to $\ell$ by
  Lemma~\ref{lemma: nonzero trace}. Thus the conductor of
  $\Z[\alpha]$, which divides the square part of the discriminant of
  $\alpha$, is also coprime to $\ell$.
\end{proof}

\begin{lemma}\label{lemma: ell splits}
Suppose that $\alpha \in \End(E(j))$ corresponds to a cycle
$\{a_1,\ldots,a_e\}$ in
$G(p,\ell)$ with no backtracking and such that $a_1$ is not dual to
$a_e$. Let $K=\QQ(\alpha)$. Then $\ell$ splits completely in $\OO_K$
as $\ell\OO_K = \pp_1\pp_2$, and $\alpha \Z[\alpha] = (\pp_i\cap \Z[\alpha])^e$ for
$i=1$ or $2$. 
\end{lemma}

\begin{proof}

Since $a_1$ is not dual to $a_e$, the conductor of $\Z[\alpha]$ is
coprime to $\ell$ by Lemmas~\ref{lemma: nonzero trace} and~\ref{lemma:
  ell-maximal}. Let $\pp$ be a prime of $\OO_K$ above $\ell$. If
$\ell$ ramifies in $K$, then the factorization $\alpha^2\OO_K =
\pp^{2e} = \ell^e\OO_K$ implies that $\alpha^2\Z[\alpha] =
\ell^e\Z[\alpha]$. But then $\alpha^2 = [\ell^e]\gamma$ for some
$\gamma\in \Z[\alpha]\subseteq \End(E(j))$. On the other hand,
$\alpha^2$ must be primitive because  the assumptions that $\alpha$ is primitive and $a_e$
is not dual to $a_1$ imply that the cycle for $\alpha^2$ has no
backtracking. This implies that $\ell$ cannot ramify in $K$. If $\ell$ is inert, it follows that
$e$ is even and 
$\alpha\Z[\alpha]=\ell^{e/2}\Z[\alpha]$, again contradicting the
assumption that $\alpha$ is primitive. 

We conclude that $\ell$ must
split completely in $K$, so let $\ell\OO_K = \pp_1\pp_2$ be the
factorization of $\ell\OO_K$. We now claim that the ideal
$\alpha\OO_K$ factors as
$\alpha\OO_K= \mathfrak{p}_1^{e}$ or $\alpha\OO_K= \mathfrak{p}_2^{e}$. 

 If the claim does not hold, then $\alpha\OO_K =
\mathfrak{p}_1^r\mathfrak{p}_2^s$ with $r,s >0$. Without loss of
generality we may assume that $r >s$. Then $\alpha\OO_K =
(\mathfrak{p}_1\mathfrak{p}_2)^s(\mathfrak{p}_1)^{r-s}=
(\ell)^s(\mathfrak{p}_1)^{r-s}$. Then in $\Z[\alpha]$, we have the
factorization
\[
\alpha\Z[\alpha] = (\ell)^s(\mathfrak{p}_1\cap\Z[\alpha])^{r-s}.
\]
This implies that $\alpha = [\ell] \gamma$ for some
$\gamma\in \End(E(j))$, but by Lemma~\ref{lemma:nobacktracking}, this
contradicts the assumption that $\alpha$ has no backtracking. \end{proof}

%\begin{remark}
%If $\ell\not=2$ and $\alpha$ is an endomorphism of $E(j)$
%corresponding to a cycle with no backtracking, then the conductor of
%$\ZZ[\alpha]$ is coprime to $\ell$ if and only if
%$\tr(\alpha)\not\equiv 0 \pmod{\ell}$. 
%\end{remark}

\begin{theorem}\label{thm:independent endos}
  Suppose that two cycles with no backtracking 
 pass through
  $E(j)$ and that at least one cycle satisfies the hypotheses of
  Lemma~\ref{lemma: ell splits}. Denote the corresponding endomorphisms of $E(j)$ by
  $\alpha,\beta$. Suppose further that $\alpha$ and $\beta$
  commute. Then there is a third cycle with no backtracking passing
  through $E(j)$ which corresponds to an endomorphism
  $\gamma\in \End(E(j))$ and two automorphisms $u,v\in \Aut(E(j))$
  which commute with $\gamma$ such that $\alpha=u\gamma^a$ and either
  $\beta=v\gamma^b$ or $\beta=v\widehat{\gamma^b}$. In particular, the
  cycle for $\alpha$ is just the cycle for $\gamma$ repeated $a$
  times, and the cycle for $\beta$ is the cycle for $\gamma$ or
  $\widehat{\gamma}$ repeated $b$ times.
\end{theorem}

\begin{proof}
  Assume that the cycle for $\alpha$ satisfies the assumption that its
  first edge is not dual to its last edge. Then the conductor of
  $\Z[\alpha]$ is coprime to $\ell$.  Since $\alpha$ and $\beta$
  commute, we must have $\beta\in \QQ(\alpha)$. Let $\OO$ be the order
  of $\QQ(\alpha)$ whose conductor is the greatest common divisor of
  the conductors of $\Z[\alpha]$ and $\Z[\beta]$. Then
  $\OO= \Z[\alpha]+\Z[\beta]\subseteq
  \End(E(j))$ and
 %$\OO\supseteq \Z[\alpha]$ and
 % $\OO\supseteq \Z[\beta]$, 
 the conductor of $\OO$ is coprime to
  $\ell$. By Lemma~\ref{lemma: ell splits}, $\ell$ splits
  completely in $K$; let $\pp_1,\pp_2$ be the primes above
  $\ell$. Then without loss of generality, we have the factorization
  $\alpha \OO = (\pp_1 \cap \OO)^i$ of $\alpha\OO$ into primes of
  $\OO$ by the same argument as in Lemma~\ref{lemma: ell splits}.
%Then if $\mathfrak{p}$ is ramified or inert, the factorization
 % $(\alpha)=\mathfrak{p}^e$ implies that $(\alpha)$ is divisible by $(\ell)$
 % and hence $\alpha$ factors through $[\ell]:E(j)\to E(j)$. This is
  %impossible by Lemma~\ref{lemma:nobacktracking}. Thus $(\ell)$ splits completely in $K$; denote the factorization by $(\ell)=\mathfrak{p}_1\mathfrak{p}_2$.
  
  %{\bf Claim:} 

%{\bf Proof of claim:} 

Observe that since $\beta\in \OO$, 
$\widehat{\beta} = \tr(\beta) - \beta \in \OO$. After possibly exchanging $\beta$ with its dual
$\hat\beta$, we get with the same argument that $\beta\OO =
(\mathfrak{p}_1\cap \OO)^j$. 
%Since $\alpha $
 % and $\beta $ commute, we have $\beta\in \Z[\alpha]\subset
 % \calO_K$.
 Now let $d=\gcd(i,j)$, which implies that there exist $m,n\in \Z $ such
  that $d=im+jn$ and hence
\[
(\mathfrak{p_1}\cap \OO)^d = (\mathfrak{p_1}\cap\OO)^{im+jn}=\alpha^m\beta^n\OO. 
\]
Set $\gamma=\alpha^m\beta^n\in K$. Then 
\[
\gamma\OO = (\mathfrak{p}_1\cap \OO)^d
\]
implies $\gamma\in \OO$ and that $\gamma$ must be primitive.
Filtering the kernel of $\gamma $ yields a cycle. Write $i=da$ and
$j=db$ for $a,b\in \N $. Then $\gamma^a\OO=\alpha\OO$, so there exists
$u\in\OO^*\subseteq \Aut(E(j))$ such that $\alpha = u\gamma^a$. We see
that the cycle for $\alpha $ is just the cycle for $\gamma $ repeated
$a$ times. Similarly, the cycle for $\beta $ is the cycle for
$\gamma $ repeated $b$ times.
\end{proof} 

We can state a more general result about when two cycles can give rise to
commuting endomorphisms. %If $p=1 \pmod{12}$, this completely
%classifies such cycles. 

\begin{corollary}
  Suppose that $P=\{a_1,\ldots,a_m\}$ is a path in $G(p,\ell)$ without
  backtracking between $E(j)$ and $E(j')$ which does not pass through
  $E(0)$ or $E(1728)$. Suppose that
  $C = \{a_{m+1},\ldots,a_{m+e}\}$ is a cycle beginning at $E(j')$ satisfying the
  assumptions of Lemma~\ref{lemma: ell splits}. Let $\widehat{P}$ be a
  path
  $\{a_{m+e+1},\ldots,a_{2m+e}\}$ without backtracking such that $a_{k}$ is dual to
  $a_{m+e+k}$ for $1\leq k \leq m$. Let $\alpha\in \End(E(j))$ be the
  endomorphism corresponding to the cycle $\{a_1,\ldots,a_{2m+e}\}$,
  the concatenation of $P$, $C$, and $\widehat{P}$.  Now let $\beta$ be the endomorphism for another cycle in $G(p,\ell)$ without
  backtracking which starts at $E(j)$, and assume $\alpha$ and $\beta$ commute. Then
    there exist automorphisms $u_1,u_2\in
    \Aut(E(j))$, an $\ell$-power isogeny $\phi: E(j)\to
    E(j')$, an endomorphism $\gamma\in
    \End(E(j'))$, automorphisms $v_1,v_2\in \Aut(E(j'))$ which commute
    with $\gamma$, and positive integers $a,b$ such that $ \alpha = u_1
    \circ \widehat{\phi} \circ v_1\gamma^a\circ \phi$ and $\beta = u_2
    \circ \widehat{\phi}\circ v_2\gamma^b\circ \phi$ or $\beta =u_2 \circ
    \widehat{\phi}\circ v_2\widehat{\gamma}^b \circ \phi$.
\end{corollary}

\begin{proof}
  %By Theorem~\ref{thm:independent endos} we can assume that $a_1$ is dual to $a_e$ and $a_1'$ is dual to
 % $a_{f}'$. There exists $m<e$ such that $a_{m}$ is dual to
 % $a_{e-m+1}$ but $a_{m+1}$ is not dual to $a_{e-m}$. Let $\phi_k$
 % correspond to $a_k$ for $1\leq k \leq e$ and set
 % $\phi:=\phi_m\circ\cdots\circ\phi_1$. the cycle
 % $\{a_{m+1},\ldots,a_{m-e}\}$ has no backtracking and its first edge
 % is not dual to its last edge. Let its corresponding endomorphism be
 Let $\alpha' \in \End(E(j'))$ be the cycle corresponding to $C$. Let the cycle for $\beta$ be $\{a_1',\ldots,a_{e'}'\}$. We can assume there is a positive integer $n$ such that $a_k'$ is dual to $a_{e'-k+1}'$ for $1 \leq k \leq n$, but $a_{n+1}'$ is not dual to $a_{e'-n}'$. We can assume such an index exists because if not, $a_1'$ is not dual to $a_{e'}'$, and we could then apply the previous theorem to $\beta$. Write $f=e'-2n$. Then we must have $f\geq 1$, because otherwise $\beta$ will not be primitive. We then have two cases to consider: the cycle $\{a_{n+1}',\ldots,a_{n+f}'\}$ satisfies the assumptions of Lemma~\ref{lemma: ell-maximal}, or $f=1$ and $a_{n+1}'$ is a self-loop which is dual to itself. We begin by considering the first case. We can assume that $m\leq n$, because otherwise we could swap the roles of $\alpha$ and $\beta$. 
 
%Similarly there exists $n$ such that
 % $a_{k}'$ is dual to $a_{f-n+1}'$ for $1\leq k \leq n$ but
  %$a_{n+1}'$ is not dual to $a_{f-n}$. Let $\psi_k$ be the isogeny for
  %the edge for $a_k'$ for $1\leq k\leq f$ and let $\beta'$ be the
 % endomorphism corresponding to the cycle
 % $\{a_{n+1}',\ldots,a_{f-n}'\}$. Without loss of generality we assume
 % $m\leq n$.
 We will proceed by induction on $m$. Assume first that
 $m=1$. %We can assume that $a_1'$
%is dual to $a_f'$, because otherwise we can apply
%Theorem~\ref{thm:independent endos}. 
%Let $\alpha'$ be the endomorphism corresponding to the cycle
%$\{a_2,\ldots,a_{e-1}\}$ and 
Let $\beta'$ correspond to
$\{a_2',\ldots,a_{2n+f-1}'\}$. 
If the cycle $\{a_1,\ldots,a_{e+2},a_1',\ldots,a_{2n+f}'\}$ has
backtracking, it follows that $a_1'$ is dual to $a_{e+2}$. As the
path $P=\{a_1\}$ does not pass through $E(0)$ or $E(1728)$, it follows that
$\phi_{e+2} = \widehat{\psi_1}$ or $\phi_{e+2} = \widehat{\psi_1}\circ
[-1]$. Additionally, since $a_1$ is dual to $a_{e+2}$, $\phi_{e+2} = \widehat{\phi_1}$ or
$\phi_{e+2}=\widehat{\phi_1}\circ [-1]$. In any case, the equality
$\alpha\beta = \beta\alpha$ implies that
$\alpha'\beta'=\beta'\alpha'$. We can now apply
Theorem~\ref{thm:independent endos}. If $m>1$, the corollary follows by applying
the same argument to $\alpha'$ and $\beta'$. 
%by using $\alpha = [\pm 1]\widehat{\phi_1}\alpha'\phi_1$ and
%$\beta = [\pm 1] \widehat{\phi_1} \beta'\phi_1$. 

We will now show that $\beta\alpha$ can not be primitive. Assume that
$\beta\alpha$ is primitive. Then its kernel is cyclic and thus
contains a unique subgroup $H$ of order $\ell$ with  $H \subset E(j)[\ell]$. Then
$H=\ker(\phi_1)$. On the other hand, the equality
$\alpha\beta = \beta\alpha$ implies that $H=\ker(\psi_1)$, so we
conclude $\phi_1 = \psi_1$ (here we use that $\phi_1$ and $\psi_1$ are
fixed representatives of edges in $G(p,\ell)$). This contradicts the assumption that
$\beta\alpha$ is primitive, since $a_{1}$ is dual to $a_e$ and
$a_1' = a_1$. 

Now we consider the case that $f=1$; we will show that in this case, $\beta$ does not commute with $\alpha$. Consider first the case that the cycle for $\beta$ is just $\{a_1'\}$, a single self-loop which is dual to itself. Then $a_1=a_1'$ by the same argument as above, by considering whether $\alpha\beta$ is primitive or not. %If it is not primitive, then $a_1$ is dual to $a_1'$ from which we conclude that $a_1=a_1'$ (here using again that $E(0)$ and $E(1728)$ are not in the path) and if it is primitive, we conclude the same. 
Then the path $\{a_2,\ldots,a_{2m+e-1}\}$ is also a cycle beginning at $E(j)$, and its corresponding endomorphism also commutes with $\beta$. By induction we conclude then that $\beta$ also commutes with $\{a_{m+1},\ldots,a_{m+e}\}$, which is impossible by Theorem~\ref{thm:independent endos}. 

If now, in the cycle for $\beta$, we have $n<m$, we can use induction to conclude that $a_k=a_k'$ for $1\leq k\leq n$, and then reduce to the case that $\beta$ is a single self-loop dual to itself. 

Thus we conclude that $m\leq n$. Again by using $\alpha\beta=\beta\alpha$, we find that  $a_k=a_k'$ for $1\leq k \leq m$, and we can reduce to the case of Theorem~\ref{thm:independent endos}. 
%The cycle for $\alpha'$ is $\{a_2,\ldots,a_{e-1}\}$ and the cycle for
%$\beta'$ is $\{a_2',\ldots,a_{f-1}'\}$, so 
%We can repeat the same
%argument $e-1$ more times until we can apply
%Theorem~\ref{thm:independent endos} to $\alpha'$ and $\beta'$. The
%conclusion of the corollary follows.
%primitive, then the equality $\alpha\beta = \beta\alpha$ implies that
%$\alpha\beta(\ker{\alpha})=\{0\}$, so $\alpha\beta = \psi \alpha$ for
%some $\psi \in \End(E(j))$. 
\end{proof}

\begin{corollary}
  Suppose that two cycles $C_1$ and $C_2$ through $E(j)$ have no
  backtracking and that $C_1$ passes through a vertex through which
  $C_2$ does not pass. Suppose also that one cycle does not contain a
  self-loop which is dual to itself. Further assume that neither cycle
  passes through $E(0)$ or $E(1728)$. Then the corresponding
  endomorphisms in $\End(E(j))$ are linearly independent.
\end{corollary}

%%%%%%%%%%%%%%%%%%%%%%%%%%%%%%%%%%%%%%%%%%%%%%%%%%%%%%%%%%%%%%%%%%%%%%%%%%%%%%%%%%%%%%%%%%%%%%%%%%%%%%%%%%%%%
%%%%%%%%%%%%%%%%%%%%%%%%%%%%%%%%%%%%%%%%%%%%%%%%%%%%%%%%%%%%%%%%%%%%%%%%%%%%%%%%%%%%%%%%%%%%%%%%%%%%%%%%%%%%%
\section{An obstruction to generating the full endomorphism ring}\label{sec: obstructions}

%%%%%%%%%%%%%%%%%%%%%%%%%%%%%%%%%%%%%%%%%%%%%%%%%%%%%%%%%%%%%%%%%%%%%%%%%%%%%%%%%%%%%%%%%%%%%%%%%%%%%%%%%%%%%
%\subsection{One obstruction for maximality}\label{sec: obstruction nonmax}
%\begin{theorem}
%	Let $\alpha$ be a loop that goes through vertices
%	$v_1,v_2, \dots v_k$ in $G(p,\ell)$. Let $E_1, \dots, E_k$ be
%%	$\mathcal{O}_1, \dots,\mathcal{O}_k$ be a maximal orders in
%	$B_{p, \infty}$ such that $\End(E_i) \cong \mathcal{O}_i$ for all
%%	\[
%	\phi_i: B_{p, \infty} \to \End(E_i) \tensor \mathbb{Q}
%%	such that $\phi_1^{-1}(\alpha) = \cdots...= \phi_k^{-1}(\alpha)$.
%\end{theorem}
%\begin{proof}
%	Use the result on page 15 of the Ken McMurdy \cite{Mur14} paper, originally due to Waterhouse.
%\end{proof}
If $C$ is a cycle in $G(p,\ell)$ which passes through $E(j_1)$ and
$E(j_2)$, then we can view it as starting at $E(j_1)$ or $E(j_2)$ and
thus it corresponds to an endomorphism $\alpha\in \End(E(j_1))$ or
$\alpha'\in \End(E(j_2))$. This suggests the following: suppose we have two
cycles which have a path between $E(j_1)$ and $E(j_2)$ in common. Then we can view them as
endomorphisms of each vertex. These endomorphisms generate an order
$\mathcal{O}$ 
contained in
the intersection of $\End(E(j_1))$ and  
$\frac{1}{\ell^e}\hat{\phi}\End(E(j_2))\phi$ where $\phi$ corresponds
to the common path. These are two maximal orders inside of 
$\End(E(j_1))\otimes \QQ$ and thus the two cycles can not generate a maximal
order. However, this does not hold if
$\End(E(j_1))\simeq \End(E(j_2))$, i.e., $j_1$ is a Galois conjugate of $j_2$.
%by $\Gal(\mathbb{F}_{p^2}/\mathbb{F}_p)$. 
This is formalized in the
following theorem.

\begin{theorem}\label{thm:nonmaximal order}
	Suppose two cycles in $G(p,\ell)$ both contain the same path between two
        vertices $E(j_1)$ and $E(j_2)$. Let $\alpha$ and
        $\beta$ be the corresponding
        endomorphisms of $E(j_1)$. If the path between $E(j_1)$ and
        $E(j_2)$ passes through additional vertices, or if $j_1^{p}\neq j_2$,
        then $\{1,\alpha,\beta,\alpha\beta\}$ is not a basis
	for $\End(E(j_1))$.
\end{theorem}
\begin{proof}
  We can assume that $j_1^{p}\neq
  j_2$, by replacing $j_2$ with an earlier vertex in the path if necessary. Let the path from
  $E(j_1)$ to $E(j_2)$ be correspond to the isogeny $\phi: E(j_1)\to
  E(j_2)$. By assumption, we can write $\alpha=\alpha_1\phi$ and write
  $\beta=\beta_1\phi$. Let
  $\alpha'=\phi\alpha_1$ and
  $\beta'=\phi\beta_1$ be the corresponding endomorphisms of
  $E(j_2)$. Assume towards contradiction that $\langle
  1,\alpha,\beta,\alpha\beta\rangle = \End(E(j_1))$. Denote the lists 
\begin{align*}
\{x_0,x_1,x_2,x_3\} &= \{1,\alpha,\beta,\alpha\beta\} \\ 
\{y_0,y_1,y_2,y_3\} &=\{1,\alpha',\beta',\alpha'\beta'\}.
\end{align*}
We now show that
$\tr(x_i\widehat{x_j})=\tr(y_i\widehat{y_j})$ for
$i,j=0,\ldots,3.$ Observe that
$[\deg\phi](x_i\widehat{x_j})=\widehat{\phi}y_i\widehat{y_j}\phi$, so
\[
\deg(\phi)\tr(x_i\widehat{x_j}) =
\tr(\widehat{\phi}y_i\widehat{y_j}\phi).
\]
On the other hand, we 
use Lemma~\ref{lemma:traceofconjugate} to compute
\[
\tr(\widehat{\phi}y_i\widehat{y_j}\phi)=\deg(\phi)\tr(y_i\widehat{y_j}).
\]

This implies that the embedding 
\begin{align*}
\End(E(j_2)) &\hookrightarrow \End(E(j_1))\otimes \QQ\\
\rho &\mapsto \widehat{\phi} \rho \phi  \otimes \frac{1}{\deg\phi} 
\end{align*}
maps $\langle 1,\alpha',\beta',\alpha'\beta'\rangle$ to an order
isomorphic to $\End(E(j_1))$ by \cite[Corollary 4.4]{Neb98}. But this violates Deuring's
correspondence.
\end{proof}

One might conjecture that two cycles in $G(p,\ell)$ which only
intersect at one vertex $E(j)$ generate $\End(E(j))$, but the example
in the following section shows this might not be true. In particular,
there is an example of two cycles which generate an order $\OO$ which
is not maximal, but there is a unique maximal order containing
$\OO$.

%In fact, one example
%shows that no two cycles generate a maximal order, but they
%do generate an order which is contained in only one maximal order. 
%\kirsten{Maybe we need more, that if the vertices are adjacent, and
%	there are multiple edges between them, that they use the same edge?
%	Not sure.}
%\jenn{one might conjecture: if two loops intersect at exactly one
%	point, then they generate a maximal order }
%%%%%%%%%%%%%%%%%%%%%%%%%%%%%%%%%%%%%%%%%%%%%%%%%%%%%%%%%%%%%%%%%%%%%%%%%%%%%%%%%%%%%%%%%%%%%%%%%%%%%%%%%%%%%

%%%%%%%%%%%%%%%%%%%%%%%%%%%%%%%%%%%%%%%%%%%%%%%%%%%%%%%%%%%%%%%%%%%%%%%%%%%%%%%%%%%%%%%%%%%%%%%%%%%%%%%%%%%%%
%%%%%%%%%%%%%%%%%%%%%%%%%%%%%%%%%%%%%%%%%%%%%%%%%%%%%%%%%%%%%%%%%%%%%%%%%%%%%%%%%%%%%%%%%%%%%%%%%%%%%%%%%%%%%

\section{Examples}\label{sec: examples}

We used the software package Magma to perform most of the computations
required to compute the endomorphism rings of supersingular elliptic
curves in characteristic $p$ with $p\in \{31,101,103\}$. In all cases we worked
with the $2$-isogeny graph. We started with the supersingular
$j$-invariants and found models for the elliptic curves $E(j)$ 
as in Equation \ref{ellcurve} that we transformed into ones of the form
$y^2=x^3+ax+b$ for some $A,B\in \F_{p^2}$. Then for every $E(j)$ we
computed the $2$-torsion points
 to generate its $2$-isogenies, as in Section~\ref{sec: defs and notation}.\\
By Theorem~\ref{thm:deuringcorrespondence} and
Lemma~\ref{lem:b_p_infty} we know that $\End(E(j))$ corresponds to a
maximal order in $B_{p,\infty}$.

For each vertex corresponding to $E(j)$, we select cycles in the
$2$-isogeny graph that satisfy the conditions of
Theorem~\ref{thm:independent endos} and compute their traces and
norms. Then we find elements of $B_{p,\infty}$ with these traces and norms
and verify that they generate a maximal order.

	\begin{example}[$p=31$]\label{ex:p=31}
		Let $p=31$. The unique quaternion algebra ramified at $p$ and $\infty$ is 
		
		\begin{equation*}
			B_{31,\infty}=\Q+\Q i+\Q j + \Q ij,
		\end{equation*}
		where $i^2=-1$ and $j^2=-31$.

		There are three $j$-invariants corresponding to
                isomorphism classes of supersingular elliptic curves
                over $\F_{p^2}$, namely $2$, $4$ and $23$. Figure~\ref{fig:isogenygraph_31_labeled} shows the $2$-isogeny graph with labeled
                edges.

		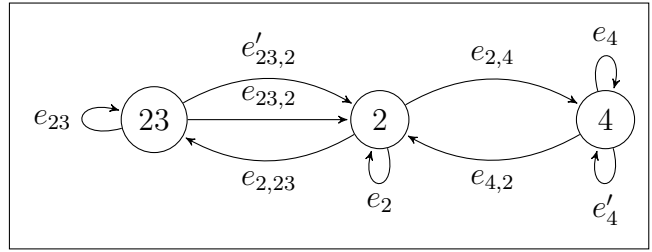
\begin{figure}[H]
			\centering
			\begin{tikzpicture}[->,>=stealth',shorten >=1pt,node distance=1.8cm, show background rectangle,auto ]
			\node[circle, minimum width=22pt,draw ](23) at (0,0) {$23$};
			\node[circle, minimum width=22pt,draw](2) at (3,0) {$2$};
			\node[circle, minimum width=22pt,draw](4) at (6,0) {$4$};
			
			\path
			(23) edge [loop left] node {$e_{23} $}(23)
			edge [bend left] node {$e_{23,2}' $}(2)
			edge  node {$e_{23,2} $}(2)
			(2) edge [bend left] node {$e_{2,23} $}(23)
			edge [loop below] node {$e_{2}$}(2)
			edge [bend left] node{$e_{2,4} $} (4)
			(4) edge [bend left] node[label=]{$e_{4,2} $} (2)
			edge [loop above] node{$e_{4} $} (4)
			edge [loop below] node{$e_{4}' $} (4)
			;

			\end{tikzpicture}
			\caption{$2$-isogeny graph for $p=31$.}
			\label{fig:isogenygraph_31_labeled}
		\end{figure}

		Table~\ref{table:p=31} contains, for each vertex, two
                cycles that correspond to elements that generate a
                maximal order in $B_{b,\infty}$. Hence these two
                cycles must generate the full endomorphism ring.

		\begin{table}[H]
			\begin{tabular}{ccccc}
				%	\hline 
				\textbf{Vertex} & \textbf{Cycle} & \textbf{Trace} & \textbf{Norm} \\ 
				\hline 
				\multirow{2}{*}	{$2$} & $e_2$ & $0$ & $2$   \\ 
				
				& $e_{2,4}e_4e_{4,2}$ & $2$ & $8$ \\ 
				\hline 
				\multirow{2}{*}	{$4$} &  $e_4$ & $1$ & $2$  \\ 
				
				& $e_{4,2}e_2e_{2,4}$ & $0$ & $8$&\\ 
				\hline 
				\multirow{2}{*}{$23$}		&$e_{23}$  & $2$ & $2$ \\ 
				&$e_{23,2}e_2e_{2,23}$ &$-1$  &$8$  \\ \hline		
			\end{tabular} \caption{}
			\label{table:p=31}
		\end{table}

		With this data we are able to generate the maximal orders that correspond to each endomorphism ring:
		\begin{align*}
			&\End(E({23}))\cong \left\langle1, -i,  -\frac{1}{2}i+\frac{1}{2}ij,\frac{1}{2}-\frac{1}{2}j \right\rangle,\\	
			&\End(E(2))\cong \left\langle 1,\frac{1}{4}i \frac{1}{4}ij,2i,\frac{1}{2}-\frac{1}{2}j\right\rangle,\\
			&\End(E(4)) \cong \left\langle 1, \frac{1}{2}+\frac{1}{6}i+\frac{1}{6}j-\frac{1}{6}ij,\frac{5}{6}i+\frac{1}{3}j+\frac{1}{6}ij,-\frac{13}{6}i+\frac{1}{3}j+\frac{1}{6}ij \right\rangle.\\
		\end{align*}

	\end{example}

	\begin{example}	[$p=103$]\label{ex:p=103}
		
		Let $p=103$. The unique quaternion algebra ramified at $p$ and $\infty$ is 
		
		\begin{equation*}
			B_{103,\infty}=\Q+\Q i+\Q j + \Q ij,
		\end{equation*}
		where $i^2=-1$ and $j^2=-103$.

		The supersingular $j$-invariants over $\F_{p^2}$ are
                $23,24,69,34,80$, and four defined over
                $\F_{p^2}-\F_p$: $\alpha$, $\beta$ and their
                conjugates. Figure~\ref{fig:isogenygraph_103_labeled}
                shows the $2$-isogeny graph.
		\begin{figure}[H]
			\centering
			\begin{tikzpicture}[->,>=stealth',shorten >=2pt,show background rectangle]
			\node[circle, minimum width=22pt,draw ](80) at (-1,0) {$80$};
			\node[circle, minimum width=22pt,draw](24) at (-1,4) {$24$};
			\node[circle, minimum width=22pt,draw](23) at (2,2) {$23$};
			\node[circle, minimum width=22pt,draw](69) at (4,2) {$69$};
			\node[circle, minimum width=22pt,draw](34) at (6,2) {$34$};
			\node[circle, minimum width=22pt,draw](a) at (9,4) {$\alpha$};
			\node[circle, minimum width=22pt,draw](a1) at (9,0) {$\overline{\alpha}$};
			\node[circle, minimum width=22pt,draw](b) at (12,4) {$\beta$};
			\node[circle, minimum width=22pt,draw](b1) at (12,0) {$\overline{\beta}$};

			\path
			(24)	edge [loop below] node[left]{$e_{24} $} ()
			edge [loop left] node[above]{$e_{24}'$}()
			(80) 	edge [loop left] node[below]{$e_{80}$}()
			edge [bend left=40] node[rotate= 33,midway,above]{$e_{80,23}'$}(23)
			(69)	edge [loop above] node[]{$e_{69}$}()
			(b)		edge [bend left=40] node[above,rotate=-90] {$e_{\beta,\overline{\beta}}'$}(b1)
			(b1) 	edge [bend left=40] node[below, rotate=-90] {$e_{\overline{\beta},\beta}'$}(b)			
			;
			%	\draw[transparent] (a1)--node [below, opaque, rotate=-50]{$e_{57,\alpha} $}(57);
			\DoubleLine{b}{b1}{->}{}{<-}{}
			\draw[transparent] (b)-- node[opaque,above, rotate=-90]{$e_{\beta,\overline{\beta}}$} (b1); 
			\draw[transparent] (b)-- node[opaque,below, rotate=-90]{$e_{\overline{\beta},\beta}$} (b1); 
			\DoubleLine{a}{a1}{->}{}{<-}{}
			\draw[transparent] (a)-- node[above, opaque,rotate=-90]{$e_{\alpha,\overline{\alpha}}$} (a1); 
			\draw[transparent] (a)-- node[below, opaque,rotate=-90]{$e_{\overline{\alpha},\alpha}$} (a1);
			\DoubleLine{24}{23}{->}{}{<-}{}
			\draw[transparent] (24)-- node[above, opaque,rotate=-33]{$e_{24,23}$} (23);
			\draw[transparent] (24)-- node[below, opaque,rotate=-33]{$e_{23,24}$} (23);
			\DoubleLine{80}{23}{->}{}{<-}{}
			\draw[transparent] (80)-- node[above, opaque,rotate=33]{$e_{80,23}$} (23);
			\draw[transparent] (80)-- node[below, opaque,rotate=33]{$e_{23,80}$} (23);
			\DoubleLine{23}{69}{->}{$e_{23,69}$}{<-}{$e_{69,23}$}
			
			\DoubleLine{69}{34}{->}{$e_{69,34}$}{<-}{$e_{34,69}$}
			
			\DoubleLine{a}{34}{->}{}{<-}{}
			\draw[transparent] (a)-- node[below, opaque,rotate=30]{$e_{\alpha,34}$} (34);
			\draw[transparent] (a)-- node[above, opaque,rotate=30]{$e_{34,\alpha}$} (34);
			\DoubleLine{a}{b}{->}{$e_{\alpha,\beta}$}{<-}{$e_{\beta,\alpha}$}
			
			\DoubleLine{34}{a1}{->}{}{<-}{}
			\draw[transparent] (34)-- node[below, opaque,rotate=-30]{$e_{\overline{\alpha},34}$} (a1);
			\draw[transparent] (34)-- node[above, opaque,rotate=-30]{$e_{34,\overline{\alpha}}$} (a1);
			\DoubleLine{a1}{b1}{->}{$e_{\overline{\alpha},\overline{\beta}}$}{<-}{$e_{\overline{\beta},\overline{\alpha}}$}
			%			\draw[transparent] ()-- node[opaque,rotate=]{} ();
			%			\draw[transparent] ()-- node[opaque,rotate=]{} ();
			% \DoubleLine{4}{2}{<-}{}{->}{}
			\end{tikzpicture}
			\caption{$2$-isogeny graph for $p=103$.}
			\label{fig:isogenygraph_103_labeled}
		\end{figure}

		After several computations, we were able to find
                generators for the maximal orders corresponding to all
                the endomorphism rings of supersingular curves $E(j)$ where
                $j\in \F_{{103^2}}$. Table~\ref{table:p=103} contains, for
                each such vertex two cycles that correspond to
                elements that generate the maximal order.

		\begin{table}[H]
			\begin{tabular}{ccccc}
				%	\hline 
				\textbf{Vertex} & \textbf{Cycle} & \textbf{Trace} & \textbf{Norm} \\ 
				\hline 
				\multirow{2}{*} {$34$}	 & $e_{34,\overline{\alpha}}e_{\overline{\alpha},\alpha}e_{\alpha,34}$& $-3$ & $8$ \\ 
				&  $e_{34,69}e_{69}e_{69,34}$ & $0$ & $8 $ \\ 
				\hline 
				
				\multirow{2}{*}	{$69$} &  $e_{69}$ & $0$ & $2$ \\ 
				& $e_{69,34}e_{34,\alpha}e_{\alpha,\overline{\alpha}}e_{\overline{\alpha},34}e_{34,69}$ & $-6$ & $ 32$ \\ 
				\hline 
				
				\multirow{2}{*} {$23$}	 & $e_{23,24}e_{24}e_{24,23}$ & $2$ & $8$ \\ 
				& $e_{23,80}e_{80}e_{80,23}$ & $-4$ & $ 8$ \\ 
				\hline 
				
				\multirow{2}{*}{$80$}	 & $e_{80}$ & $2$ & $2$ \\ 
				&  $e_{80,23}e_{23,69}e_{69}e_{69,23}e_{23,80}$ & $0$ & $ 32$ \\ 
				\hline 
				
				\multirow{2}{*}{$24$}	 & $e_{24}$ & $-1$ & $2$\\ 
				&  $e_{24,23}e_{23,69}e_{69}e_{69,23}e_{23,24}$ & $0$ & $32 $ \\ 
				\hline 		
				
			\end{tabular} \caption{}
\label{table:p=103}
		\end{table}
		In the case of the vertex $\alpha$, we found an example of two cycles that do not share an additional vertex but that do not generate a maximal order. For instance, the cycles
		\begin{align*}
			&e_{\alpha,\beta}e_{\beta,\overline{\beta}}'e_{\overline{\beta},\beta}'e_{\beta,\alpha}\\
			&e_{\alpha,34}e_{34,69}e_{69}e_{69,34}e_{34,\alpha}
		\end{align*}
			generate the order
                $\mathcal{O}=\left\langle 1, -\frac{1}{2} +
                  \frac{17}{6}i - \frac{1}{6}j + \frac{1}{6}ij,
                  -\frac{5}{2}i + \frac{1}{2} ij, -\frac{1}{2} -
                  \frac{22}{3} i - \frac{11}{6} j - \frac{2}{3} ij
                \right\rangle$ and there is a unique maximal order
                containing it, hence this corresponds to
                $\End(E({\alpha}))\cong\End(E({\overline{\alpha}}))$. Finally,
                there is only one maximal order remaining in
                $B_{103,\infty}$, which by
                Theorem~\ref{thm:deuringcorrespondence} is isomorphic
                to the endomorphism rings of $E({\beta})$ and
                $E({\overline{\beta}})$.
		
		The endomorphism rings are then isomorphic to the following maximal orders:
		\begin{align*}
			&	\End(E({80}))\cong\left\langle 1,i, \frac{1}{2}i+\frac{1}{2}ij,\frac{1}{2}+\frac{1}{2}j  \right\rangle, \\
			&	\End(E({23}))\cong\left\langle 1, 2i,\frac{3}{4}i+\frac{1}{4}ij,\frac{1}{2}-\frac{1}{2j} \right\rangle, \\
			&	\End(E({34}))\cong\left\langle 1, \frac{17}{14}i+\frac{1}{14}ij  \frac{15}{7}i-\frac{2}{7}ij,\frac{1}{2}-\frac{1}{2}j\right\rangle, \\
			&	\End(E({69}))\cong\left\langle 1,\frac{1}{2}+\frac{1}{7}i+\frac{3}{14}j,\frac{1}{2}-\frac{16}{7}i+\frac{1}{14}j,\frac{1}{2}-\frac{17}{14}i-\frac{1}{14}j-\frac{1}{2}ij\right\rangle, \\
			&	\End(E({24}))\cong\left\langle 1,\frac{1}{2}+\frac{3}{8}i+\frac{1}{8}ij,\frac{1}{2}-\frac{29}{8}i+\frac{1}{8}ij,-\frac{13}{8}i+\frac{1}{2}j+\frac{1}{8}ij\right\rangle. \\
			&	\End(E({\alpha}))\cong \left\langle -1,-\frac{1}{2}+\frac{1}{6}i-\frac{1}{6}j-\frac{1}{6}ij,3i, \frac{5}{6}i-\frac{1}{3}j+\frac{1}{6}ij   \right\rangle,\\
			&	\End(E({\beta}))\cong\left\langle 1, \frac{1}{2}+\frac{13}{10}i+\frac{1}{10}j-\frac{1}{10}ij, -\frac{12}{5}i+\frac{1}{5}j-\frac{1}{5}ij,\frac{1}{2}-\frac{3}{5}i+\frac{3}{10}j+\frac{1}{5}ij   \right\rangle.\\
		\end{align*}

	\end{example}

	\begin{example}	[$p=101$]\label{ex:p=101}
		
		Let $p=101$. The unique quaternion algebra ramified at $p$ and $\infty$ is 
		
		\begin{equation*}
			B_{101,\infty}=\Q+\Q i+\Q j + \Q ij,
		\end{equation*}
		where $i^2=-2$ and $j^2=-101$.

		The supersingular $j$-invariants over $\F_{101^2}$ are
                $64,0,21,57,3,59,66$, and two additional ones, which
                we denote by $\alpha$ and $\overline{\alpha}$, are
                defined over
                $\F_{p^2}-\F_p$. Figure~\ref{fig:isogenygraph_101_labeled}
                shows the $2$-isogeny graph.
		
		\begin{figure}[H]
			\centering
			\begin{tikzpicture}[->,>=stealth',shorten >=2pt,show background rectangle]
			
			%\draw[help lines] (0,0) grid (15,6);
			
			\node[circle, minimum width=22pt,draw ](0) at (0,3) {$0$};
			\node[circle, minimum width=22pt,draw ](66) at (2.5,3) {$66$};
			\node[circle, minimum width=22pt,draw ](21) at (5,3) {$21$};
			\node[circle, minimum width=22pt,draw ](57) at (7.5,3) {$57$};
			\node[circle, minimum width=22pt,draw ](64) at (10,3) {$64$};
			\node[circle, minimum width=22pt,draw ](3) at (12.5,3) {$3$};
			\node[circle, minimum width=22pt,draw ](59) at (15,3) {$59$};
			\node[circle, minimum width=22pt,draw](a1) at (5,6) {$\alpha$};
			
			\node[circle, minimum width=22pt,draw](a2) at (5,0) {$\overline{\alpha}$};
			
			%\node[rotate=50](label1) at (3.3,4.6)  {$e_{66,\overline{\alpha}}$};
			%rotated 
			\draw[transparent] (66)--node [above, opaque, rotate=50]{$e_{66,\alpha}$}(a1);
			\draw[transparent] (66)--node [below, opaque,rotate=50]{$e_{\alpha,66}$}(a1);
			
			\draw[transparent] (a1)--node [above, opaque,  rotate=-50]{$e_{\alpha,57} $}(57);
			\draw[transparent] (a1)--node [below, opaque, rotate=-50]{$e_{57,\alpha} $}(57);
			
			\draw[transparent] (a1)--node [above, opaque, rotate=-90]{$e_{\alpha,21} $}(21);
			\draw[transparent] (a1)--node [below, opaque, rotate=-90]{$ e_{21,\alpha}$}(21);
			
			\draw[transparent] (21)--node [above, opaque, rotate=-90]{$e_{21,\overline{\alpha}} $}(a2);
			\draw[transparent] (21)--node [below, opaque, rotate=-90]{$e_{\overline{\alpha},21} $}(a2);
			
			\draw[transparent] (a2)--node [above, opaque, rotate=50]{$e_{\overline{\alpha},57} $}(57);
			\draw[transparent] (a2)--node [below, opaque, rotate=50]{$e_{57,\overline{\alpha}} $}(57);
			
			\draw[transparent] (a2)--node [above, opaque, rotate=-50]{$e_{66,\overline{\alpha}} $}(66);
			\draw[transparent] (a2)--node [below, opaque, rotate=-50]{$e_{\overline{\alpha},66} $}(66);
			
			\DoubleLine{0}{66}{->}{$e_{0,66}'$}{->}{$e_{0,66}$}
			\DoubleLine{66}{a1}{->}{}{<-}{}
			\DoubleLine{66}{a2}{->}{}{<-}{}
			\DoubleLine{a1}{57}{->}{}{<-}{}
			\DoubleLine{a2}{57}{->}{}{<-}{}
			\DoubleLine{a1}{21}{->}{}{<-}{}
			\DoubleLine{a2}{21}{->}{}{<-}{}
			\DoubleLine{57}{64}{->}{$e_{57,64}$}{<-}{$e_{64,57}$}
			\DoubleLine{64}{3}{->}{$e_{64,3}$}{<-}{$e_{3,64}$}
			\DoubleLine{3}{59}{->}{$e_{3,59}$}{<-}{$e_{59,3}$}
			\path
			
			(0) 	edge 	[bend left=50] node[above] {$e_{0,66}''$} (66)
			(66) 	edge 	[bend left=50] node[below] {$e_{66,0}$} (0)		
			
			(21) edge [loop right] node[right]{$e_{21}$}() 
			(59) edge [loop below] node[below]{$e_{59}'$}() 
			edge [loop above] node[above]{$e_{59}$}() 
			
			(64) edge [bend left=50] node[above] {$e_{64,3}'$}(3)
			(3)		edge [bend left=50] node [below] {$e_{3,64}'$}(64)
			;
			
			\end{tikzpicture}
			\caption{$2$-isogeny graph for $p=101$.}
			\label{fig:isogenygraph_101_labeled}
		\end{figure}
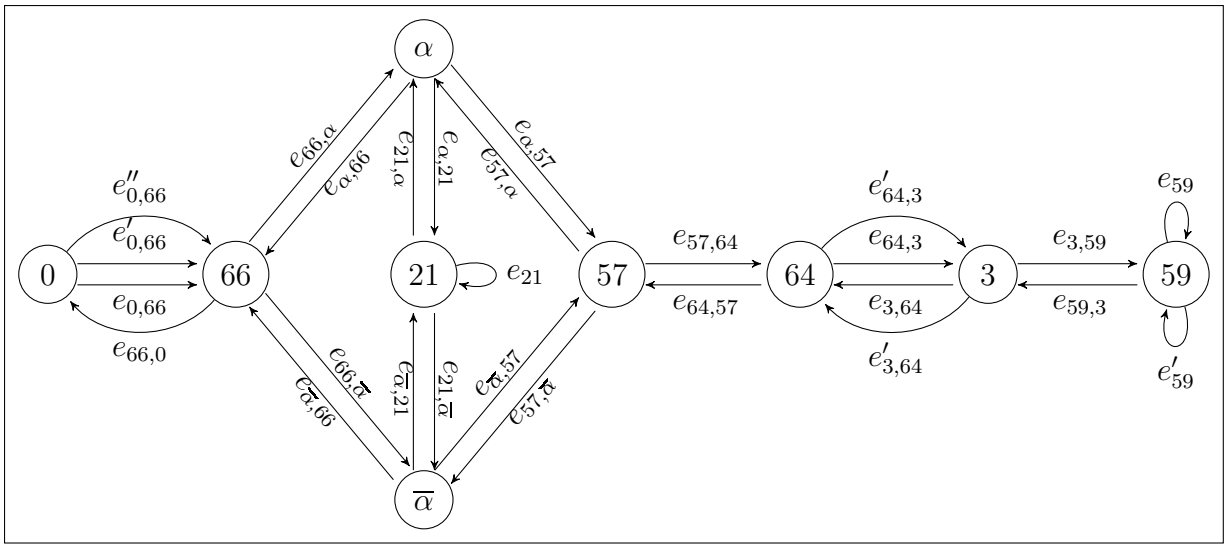

		It was possible to find two cycles that generate the
                maximal order corresponding to $\End(E(j))$ where
                $j \in \{3, 59, 64, 66 \}$. Table~\ref{table:p=101}
                contains the data for these cycles.

		\begin{table}[H]
			\begin{tabular}{ccccc}
				%	\hline 
				\textbf{Vertex} & \textbf{Cycle}  & \textbf{Trace} & \textbf{Norm} \\ 
				\hline 
				\multirow{2}{*} {$ 3$}		  & $ e_{3,59}e_{59}e_{59,3}$& $ 2$ & $8$ \\ 
				  & $e_{3,64}e_{64,3}' $ & $-1 $ & $ 4$ \\ 
				\hline 
				\multirow{2}{*} {$ 59$}		 &  $ e_{59}$& $-1$ & $2$ \\ 
				&  $ e_{59,3}e_{3,64}e_{64,3}e_{3,59}$ & $ -8$ & $16 $ \\ 
				\hline 
				\multirow{2}{*} {$64 $}		  & $e_{64,57}e_{57,\alpha}e_{\alpha,66}e_{66,\overline{\alpha}}e_{\overline{\alpha},57}e_{57,64} $& $ 10$ & $64 $ \\ 
				 & $ e_{64,3}e_{3,64}'$ & $ -1$ & $4 $ \\ 
				\hline 
				
				\multirow{2}{*} {$ 66$}		 & $e_{66,0}e_{0,66} $& $ 2$ & $4 $ \\ 
				&  $e_{66,\alpha}e_{\alpha,57}e_{57,\overline{\alpha}}e_{\overline{\alpha},66}$ & $ 5$ & $16 $ \\ 
				\hline

			\end{tabular} \caption{}
			\label{table:p=101}
		\end{table}
		
		For the vertices $21,57,\alpha$ no two cycles
                were found that generate the full endomorphism
                ring. However, in each of these cases we were able to
                to generate an order from two cycles which happened to
                be contained in a unique maximal order. These cycles
                are listed in Table~\ref{table:p=101_notmax}.

		\begin{table}[H]
			\begin{tabular}{cccccc}
				%	\hline 
				\textbf{Vertex} & \textbf{Cycle} &  \textbf{Trace} & \textbf{Norm} \\ 
				\hline 
				\multirow{2}{*} {$21 $}		 &  $ e_{21}$& $0 $ & $2$ \\ 
				&  $e_{21,\alpha}e_{\alpha,66}e_{66,0}e_{0,66}'e_{66,\alpha}e_{\alpha,21} $ & $-8 $ & $64 $ \\ 
				\hline 
				\multirow{2}{*} {$57$}		 &  $e_{57,64}e_{64,3}e_{3,59}e_{59}e_{59,3}e_{3,64}e_{64,57} $& $ -8$ & $128$ \\ 
				&   $e_{57,\alpha}e_{\alpha,66}e_{66,\overline{\alpha}}e_{\overline{\alpha},37} $& $-5$ & $16 $ \\ 
				\hline 
				
				\multirow{2}{*} {$\alpha $}		 &  $ e_{\alpha,21}e_{21}e_{21,\overline{\alpha}}e_{\overline{\alpha},57}e_{57,\alpha}$& $5 $ & $32$ \\ 
				&  $e_{\alpha,66}e_{66,0}e_{0,66}'e_{66,\alpha}$  & $ 4$ & $ 16$ \\ 
				\hline

			\end{tabular} \caption{}
			\label{table:p=101_notmax}
		\end{table}
		By Theorem~\ref{thm:nonmaximal order}, no two cycles
                through $j=0$ generate a maximal order, but it is possible
                to determine which one corresponds to the endomorphism
                ring of $E(0)$ once we ruled out the other seven. The
                endomorphism rings are then isomorphic to the
                following maximal orders:
		
		\begin{align*}
			\End(E(3)) &\cong \left\langle 1, \frac{1}{2} - \frac{13}{12}i + \frac{1}{12}ij,\frac{5}{6}i +\frac{1}{6}ij, \frac{5}{12}i -  \frac{1}{2}j +  \frac{1}{12}ij \right\rangle,\\
			\End(E({59})) &\cong \left\langle   1, \frac{1}{2} + \frac{5}{12}i - \frac{1}{12}ij, -\frac{13}{6}i - \frac{1}{6}ij, -\frac{13}{12}i + \frac{1}{2}j - \frac{1}{12}ij  \right\rangle,\\
			\End(E({64}))&\cong \bigg\langle   -1,
                          -\frac{1}{2} - \frac{3}{5}i -\frac{1}{10}j
                          +\frac{1}{10}ij, -\frac{1}{2} -
                          \frac{21}{20}i + \frac{1}{5}j +
                          \frac{1}{20}ij,  \\
&\quad\quad-\frac{67}{20}i - 1/10j - 3/20ij \bigg\rangle,\\
			\End(E({66}))&\cong  \left\langle  1, \frac{7}{10}i - \frac{1}{10}ij, \frac{1}{2} - \frac{29}{20}i - \frac{3}{20}ij, \frac{7}{20}i - \frac{1}{2}j - \frac{1}{20}ij \right\rangle,\\
			\End(E({21}))&\cong\left\langle  -1, i, -\frac{1}{2} + \frac{1}{4}i - \frac{1}{4}ij, -\frac{1}{2} + \frac{1}{2}i -\frac{1}{2}j \right\rangle,\\
			\End(E({57}))&\cong \left\langle 1, \frac{1}{2} - \frac{13}{28}i + \frac{1}{7}j + \frac{1}{28}ij, -\frac{53}{28}i - \frac{1}{14}j + \frac{3}{28}ij, \frac{1}{2} - \frac{11}{4}i -\frac{1}{4}ij  \right\rangle,\\
			\End(E(0))&\cong  \left\langle  -1, -\frac{1}{2} + \frac{7}{20}i + \frac{1}{20}ij, -\frac{1}{2} + \frac{9}{5}i + \frac{1}{2}j - \frac{1}{10}ij, -\frac{29}{20}i + \frac{1}{2}j + \frac{3}{20}ij \right\rangle,\\
			\End(E(\alpha))&\cong\End(E({\overline{\alpha}}))  \cong \left\langle -1, 2i, -\frac{1}{2} + \frac{3}{8}i + \frac{1}{4}j - \frac{1}{8}ij, -\frac{7}{8}i + \frac{1}{4}j + \frac{1}{8}ij \right\rangle.\\
		\end{align*}

	\end{example}
%\kirsten{I will try to draw the graphs for our examples $p=103$ and $p=31$}.
%\jenn{In $p=103$ case, we could not find generators for $\End(\alpha)$
%	using any two loops in the graph; is it possible to generate the
%	full endomorphism ring if we use more than two loops?}
%\subsection{Example}\label{sec: example}
%\jenn{conjecture of the previous subsection is wrong; new conjecture: if two loops intersect at exactly one point, then there is a unique maximal ideal containing both of these loops. Show this for $p=103$ and $\End(\alpha)$.}

\section*{Appendix - Modified Schoof's algorithm for traces of arbitrary endomorphisms}
Let $E$ be an elliptic curve over a finite field $\F_q$ of
characteristic $p \neq 2,3$. The Frobenius endomorphism
$\phi \in \End_{\F_q}(E)$ takes any point $(x,y) \in E(\F_q)$ to
$(x^q, y^q)$; it satisfies the relation in $\End_{\F_q}(E)$, given by
\[
\phi^2 - t\phi + q = 0.
\]
Here, $t$ is called the trace of the Frobenius endomorphism, and it is related to the number of $\F_q$-points on $E$ via the relation
\[
\#E(\F_q) = q + 1 - t.
\]

Schoof's algorithm~\cite{Schoof} computes the trace of the Frobenius
endomorphism in $O(\log^9q)$ elementary operations (bit
operations). This algorithm has been improved in~\cite{SS15} to be
completed in $O(\log^5q \log\log q)$ operations.

We outline a modification of Schoof's algorithm that computes the
trace of any endomorphism $\alpha \in \End_{\F_q}(E)$ that corresponds
to a cycle in the $\ell$-isogeny graph, where $\ell \neq p$ is a
prime. That is, we assume that we are given a path of length $e$ in an
$\ell$-isogeny graph; this path can be represented as $e$ elliptic
curves $E_1, \ldots, E_e$ over $\F_q$ in short Weierstrass form,
defined over $\F_{p^2}$, together with the coordinates of the
$\ell$-torsion points $P_1, \ldots, P_e$ in $E_1, \ldots, E_e$,
respectively, that generate the kernel of order $\ell$ giving the
respective edges on the isogeny graph. By
Corollary~\ref{cor:isogenydefined} we can take this isogeny to be
defined over an extension of degree at most degree $6$ of $\F_{p^2}$
. If $\ell=O(\log p)$ and the path has length $e=O(\log p)$, which are
the parameters that are most interesting, we will show that the trace
of this endomorphism can be computed in $\tilde{O}(\log^7 p)$ time by
using a modified version of Schoof's algorithm (as well as V\'elu's
formula, to compute the explicit equation for the endomorphism).

The na\"ive computation of the composition of the $e$ isogenies via
V\'elu's formula yields a formula for the $\ell^e$-isogeny that
requires at least $O(\ell^e)$ elementary operations; in order to cut
down on the number of elementary operations required to compute the
explicit formula for the isogeny, we note that the explicit isogeny
formula is simpler on the set of $m$-torsion points for any $m$, by
taking the quotient modulo the division polynomials. Thus,
$\ell^e$-isogenies on $E[m]$ can be computed much more quickly, and
this is sufficient information to which one can apply Schoof's idea.

\subsection{Division polynomials}

 Let $f_{k}(X)$ denote the
$k$-th division polynomial of $E$. It is a polynomial whose roots are the $x$-coordinates of the nonzero
elements of the $k$-torsion subgroup of $E$. When $k$ is coprime to
$p$, the degree of $f_k$ is
$(k^2-1)/2$.
The division polynomials can be defined recursively and the complexity
of computing them is analyzed in \cite{SS15}.
%The division polynomials $\Psi_n(x,y)$ are defined inductively as follows:
%\begin{align*}\label{equation:division}
%	\Psi_{-1}(x,y)&=-1, \quad \Psi_0(x,y)=0, \quad	\Psi_1(x,y)=1\text, \quad \Psi_2(x,y)=2y,\\
%	\Psi_3(x,y)&=3x^4+6ax^2+12bx-a^2,\\
%	\Psi_4(x,y)&=4y(x^6+5ax^4+20bx^3-5a^2x^2-4abx-a^3-8b^2),\\
%	\Psi_{2m+1}(x,y)&=\Psi_{m+2}\Psi_m^3-\Psi_{m-1}\Psi_{m+1}^3, \quad m\geq 1,\\
%	\Psi_{2m}(x,y)&= \Psi_m(\Psi_{m+2}\Psi_{m-1}^2 - \Psi_{m-2}\Psi_{m+1}^2)/2y,\quad m\geq %1.
%\end{align*}
%
%One can eliminate the variable $y$ in $\Psi_m(x,y)$ by using the
%relation $y^2 = x^3 + ax + b$ and call the resulting polynomial
%$\Psi_m'(x,y)$. Then one defines:
%\begin{align*}
%	f_m(x)&=\Psi_m'(x,y), \quad \textup{      if $m$ is odd},\\
%	f_m(x)&=\Psi_m'(x,y)/y, \quad \textup{      if $m$ is even}.
%\end{align*}
%The $f_m$ with $m \geq 2$ have the following property: if $P = (x,
%y)\in E(\F_q)$ with $P \notin E[2]$, then $mP=\mathbf{O}$ if and only
%if $f_m(x)=0$. Furthermore, one can deduce that $\deg f_m = O(m^2)$
%for $p \nmid m$. 

%\subsubsection{Complexity of computing $f_1, \dots,f_m$} 
%In this
%section, we analyze the number of elementary operations needed to
%calculate the formulae collected above. 
Let $M(n)$ denote the number
of elementary operations required to multiply two $n$-bit integers. If
we choose to multiply two $n$-bit integers via long multiplication,
then $M(n) = O(n^2)$; if we multiply two numbers using the Fast
Fourier Transform (FFT), then $M(n) = O(n \log n \log \log n)$.

\begin{proposition}\label{P: complexitydivpoly}
Given a natural number $m>1$, the division polynomials $f_1, \ldots, f_m$ can be computed in $O(m M(m^2\log q))$ time. 
\end{proposition}
\begin{proof}
Using the recursive relations defining the division polynomials, $f_k$
can be computed in $O(M(k^2\log q))$ time by using a double-and-add
method. Thus $f_1,\ldots,f_m$ can be computed in $O(mM(m^2\log q))$
time; see~\cite[Section 5.1]{SS15}. 
\end{proof}

\subsection{V\'elu's formula}\label{sec:velusFormula}
%\begin{enumerate}
%\item Finding the $x$-coordinates of the kernel (using the multiplication formula in Proposition 2.2 of Schoof's paper, find the $x$-coordinate of $P_1, 2P_1, \ldots, (\ell-1)P_1$.)
%\item give V\'elu's formula on each $\ell$-isogeny (computationally this is more effective than computing the formula for $\ell^e$ isogeny)
%\end{enumerate}

We continue to work over $\F_{q}$; typically we will work over an
extension of $\F_{p^2}$ of degree at most $6$. If $E/\F_{q}$ is an elliptic
curve, then $E[\ell] \isom \Z/\ell\Z \times \Z/\ell\Z$ is defined over
a field extension of $\F_q$ of degree $O(\ell^2)$, since the
$\ell$-th division polynomial $f_{\ell}$ has degree $(\ell^2-1)/2$.

%Let $G(p,\ell)$ be the $\ell$-isogeny graph for $p$ as defined in
%Definition~\ref{deg: isogeny
 % graph}. %\jenn{check reference after merging}
Since an $\ell$-isogeny has degree $\ell$ and $\ell$ is prime, its kernel must be
generated by a point $Q \in E[\ell]$. Thus we will represent each
directed edge $E(i) \to E(j)$ by a such a point
$Q$. %\jenn{repeated information; paragraph 3 in page 1}

In this subsection we provide a way to compute the explicit formulas
for the elements of $\End(E)$ that appear as cycles in the
supersingular $\ell$-isogeny graph in characteristic $p$; such an endomorphism is a composition of
$\ell$-isogenies, and it often suffices to know the explicit equations
for the sequence of $\ell$-isogenies. We are adapting the
work of V\'elu~\cite{Vel71}. 
%We assume that edges in $G(p,\ell)$ are represented by a point of order $\ell$ on a supersingular elliptic curve over $\F_{p^2}$. 

\subsubsection{$\ell$-isogenies}
\label{SS: isogenies}
Let $\psi: E \to E'$ be an $\ell$-isogeny, and let
$F = \ker \psi = \langle Q \rangle$ with $Q \in E[\ell]$ so that
$E' \isom E/F$. For each point $P \in E(k)$, define two functions
$X, Y \in k(E)$ by:
\begin{equation} \label{E: veluisogeny}
X(P) = x(P) + \sum_{i=1}^{\ell-1} [x(P+iQ)-x(iQ)];\quad
Y(P) = y(P) + \sum_{i=1}^{\ell-1}[y(P+iQ)-y(iQ)].
\end{equation}
Then $X(P) = X(P+R)$ for any $R \in F = \langle Q \rangle$ by
definition, and similarly for $Y$. Thus, $X, Y \in k(E')$. %We also
%check that the conditions in Lemma \ref{L: generators} hold, leading
In fact, $k(E') = k(X,Y)$. Then the isogeny $\psi: E \to E'$
given by the equation $P \mapsto (X(P),Y(P))$ is the unique
$\ell$-isogeny from $E$ to $E'$ given by the kernel $F$.

\subsubsection{Explicit equations}
Let $E: y^2 = x^3 + ax + b$ be the Weierstrass equation for $E$, and
$Q = (x(Q), y(Q)) \in E[\ell]$ be a generator of $F$. We provide an
explicit equation for $X$ and $Y$ in terms of $x$ and $y$. Instead of
computing all $\ell$ terms in the summation in \S \ref{SS: isogenies},
it is in fact easier to compute $(\ell-1)/2$ terms given by
\begin{equation}\label{E: veluisogenyxcoord}
x(P+iQ)-x(iQ)+x(P-iQ)-x(-iQ),
\end{equation}
where $i$ ranges between $1, \cdots, (\ell-1)/2$. These values are
obtained from the addition formula:%Proposition~\ref{proposition:nP-formula}:
\[
x(P+iQ)-x(iQ)+x(P-iQ)-x(-iQ) = \frac{6x(iQ)^2+2a}{(x(P)-x(iQ))^2}  + \frac{4y(iQ)^2}{(x(P)-x(iQ))^3},
\]
%\efrat{ I think it should be $6x(iQ)^2+2a$ any time this expression appears. }

and
\begin{align*}
y(P+iQ)&-y(iQ)+y(P-iQ)-y(-iQ) \\
&= \frac{-8y(P)y(iQ)^2}{(x(P)-x(iQ))^3} - \frac{(6x(iQ)^2+2a)(y(P)-y(iQ))}{(x(P)-x(iQ))^2} + \frac{(3x(iQ)^2+a)(-2y(iQ))}{(x(P)-x(iQ))^2}.
\end{align*}
%We also note that by \jenn{reference Catalina's section on division polynomials}, there exists an explicit formula calculating $x(iQ)$ and $y(iQ)$ for any integer $i$, so there is only $O(1)$ of calculations being done here. Summing over the $(\ell-1)/2$ terms gives us an explicit formula for the isogeny for $O(\ell)$ amount of calculations. \jenn{Someone check this.}

One can also compute an explicit affine equation for the elliptic
curve $E' \isom E/F : Y^2 = X^3 + AX + B$ by considering the relation
between $X$ and $Y$:
\[
A = a-5\sum_{i=1}^{(\ell-1)/2} (6x(iQ)^2+2a), \quad 
B = b-7\sum_{i=1}^{(\ell-1)/2} (4y(iQ)^2 + x(iQ)(6x(iQ)^2+2a)).
\]
%We conclude by remarking that if we wish to calculate the explicit isogeny attached to a path of length $e$ on our $\ell$-isogeny graph, given by the sequence of vertices $E_0, \ldots, E_e$, then the isogeny computation above is repeated $e$ times, leading to $O(\ell e)$ complexity. \jenn{Someone else check this, please!}

\subsubsection{Complexity of computing isogenies}
\begin{proposition}\label{P: veluisogeny}
  Let $E$ be an elliptic curve over $\F_q$ whose equation is given in
  Weierstrass form $y^2 = x^3 + ax + b$, along with
  $Q = (x(Q), y(Q)) \in E[\ell]$, a nontrivial $\ell$-torsion
  point. The explicit equation for the isogeny
  $\psi: E \to E/\langle Q \rangle$ and a Weierstrass equation for
  $E/\langle Q \rangle$ can be computed in $O(\ell^4 M(\log q))$ time.
\end{proposition}
\begin{proof}
  Using V\'elu's formula, rational functions for $\psi$ can be
  computed using $\ell$ many additions in $E[\ell]$. We can add two
  points in $E[\ell]$ in $O(\ell^3M(\log q))$ time, since we are
  working over an extension of $\F_q$ of degree $O(\ell^2)$. Thus the
  total time to compute $\psi$ is $O(\ell^4M(\log q))$. Similarly,
  using V\'elu's formulas, we can compute an equation for
  $E/\langle Q \rangle $ in $O(\ell^4 \log q)$ time.
\end{proof}

Given an $\ell$-isogeny $\psi: E \to E':=E/\langle Q
\rangle$ whose kernel is generated by some $Q \in E[\ell]$, as well as
a prime $m \neq 2,p$, we are interested in the explicit formula
for the induced isogeny on the $m$-torsion points $\psi_m: E[m] \to
E'[m]$. If $E$ is defined by the equation $y^2 = x^3 + ax + b$, and
$f_m(x)$ is the $m$-th division polynomial for $E$, then $E[m] = \Spec
\F_q[x,y]/I$, where $I = \langle f_m(x), y^2 - (x^3 + ax + b)\rangle$;
thus, we may reduce the coordinates of the explicit formula for the
isogeny $\psi$ given by $(x,y) \mapsto (X(x,y), Y(x,y))$ modulo the
ideal $I$, and the resulting map $\psi_m$ agrees with $\psi$ on
$E[m]$.

\begin{proposition}\label{P: reduction}
  Keeping the notation of the discussion in the above paragraph,
  $\deg \psi_m = O(m^2)$, and $\psi_m$ can be computed in
  $O(M(d\log q)\log d)$ elementary operations, where
  $d \in \max \{\ell, m^2\}$.
\end{proposition}
\begin{proof}
  First we estimate the degree of $\psi$: each expression
  $x(P+iQ)-x(iQ)+x(P-iQ)-x(-iQ)$ has at most a cubic denominator in
  $x$, and summing over $(\ell-1)/2$ terms means that
  $\deg X(P) \leq 3(\ell-1)/2$. Similarly, $\deg Y(P)$ is $O(\ell)$.

  Now, we may replace any appearance of $y^{2}$ with the equation of
  the elliptic curve $E: y^{2}=x^{3}+ax+b$ so that the expressions
  $x(\psi(P))$ and $y(\psi(P))$ involve only powers of $x$ and is
  possibly at most linear in $y$. Next, reduce modulo $f_{m}(x)$, so
  that that the degree of the resulting expression is bounded by
  $\deg f_{m} = O(m^2)$.

%
%We may eliminate any $Y^d$ terms for $d \geq 2$ using the equation of the elliptic curve given by $Y^2 = X^3 + AX + B$; then eliminate any $X^d$ for $d \geq \deg f_m = O(m^2)$ (This degree can be computed precisely; it is $1/2(m^2-1)$ if $m$ is odd, and $1/2(m^2-4)$ if $m$ is even). 

  Let $d$ be such that $d \in \max\{\ell, m^2\}$, so that
  $\deg x(\psi(P)), \deg y(\psi(P)), \deg f_{m} \leq O(d)$.  Then
  by~\cite[Lemma 9, p.
  315]{SS15}, %\jenn{cite something, like the chart in page 11 of Shparlinski-Sutherland}
  it takes $O(M(d\log q)\log d)$ elementary operations to compute the
  reduction of the isogeny formula modulo $f_m$.
\end{proof}

\subsection{Computing the trace on $m$-torsion points}\label{sec: traceFromNorm} 

Now we compute the trace of an endomorphism $\psi \in \End(E)$, where
$\psi$ appears as a cycle of length $e$ in the supersingular $\ell$-isogeny graph
in characteristic $p$. We present a modification of Schoof's algorithm in order to
accomplish this.

The endomorphism $\psi$ satisfies the equation
$x^2-\tr(\psi)x+\norm(\psi)$. There is a simple relationship between
$\tr(\psi)$ and $\norm(\psi)$:
\begin{lemma} \label{L: boundfortrace}
Let $\psi \in \End(E)$. Then $|\tr(\psi)| \leq 2 \norm(\psi)$.
\end{lemma}
\begin{proof} If $\psi$ is multiplication by some integer, then its
  characteristic polynomial is $x^2 \pm 2nx + n^2$, with $n \in \mathbb{N}$. Then
  $|\tr(\psi)| = 2n$, $\norm(\psi) =n^2$, and the statement of the
  lemma holds.

  If $\psi$
  is not multiplication by an integer, then $\Z[\psi]$ is an order in
  the ring of integers $\calO_K$ for some quadratic imaginary number
  field $K$. Hence we can fix an embedding $\iota: \Z[\psi]
  \hookrightarrow \calO_K$. Since $\iota(\psi)$ is imaginary, its
  characteristic polynomial $x^2 - \tr(\psi)x + \norm(\psi)$ must have
  discriminant $<0$, so $|\tr(\psi)| \leq 2\sqrt{\norm(\psi)}$. 
%In
%  either case, the approximation $|\tr(\psi)| \leq 2 \norm(\psi)$
%  clearly holds.
\end{proof}

As in Schoof's algorithm, we begin by looking for a bound $L$ such that
\begin{equation}\label{eq: N as product}
N:=\prod_{\substack{m\leq L \mbox{ prime } \\ m\neq 2,p}} m > 2\norm(\psi) = 2\ell^e,
\end{equation}
where the last equality follows from the fact that the cycle
corresponding to $\psi$ in the isogeny graph has length $e$, so
$\norm(\psi)=\ell^e$. By the Prime Number Theorem, we can take
$L=O(\log p)$ and there are $O(\log p / \log\log p)$ many primes less
than $L$. 

%\begin{proposition}
%\label{P: Lislogq}
%If $L$ is a bound such that
%\[
%\prod_{m < L} m > M,
%\]
%where the product runs over the primes $m \neq 2, p$, then $L = O(\log M)$.
%\end{proposition}
%\begin{proof}
%Let
%\[
%\pi(n) = \#\{x: x \leq n, x \textup{ is a prime.}\}.
%\]
%Then $\pi(C \log M) = O(\frac{C \log M}{\log (C\log M)})$ by the prime number theorem, so multiplying the primes in the interval $[C_1\log M, C_2 \log M]$ \travis{What are $C_1,C_2$?} gives $O(\log M^{\frac{\log M}{\log \log M}})$, which is at least $O(M)$. 
%\end{proof}
%If $\psi$ is a cycle of length $e$ in the $\ell$-isogeny graph, with
%$e=O(\log p)$, we use
%$M = 4 \ell^e$ and have $L = O(\log{\ell^e})$.

Let $m$ be a prime. Any $\psi \in \End(E)$ induces an endomorphism $\psi_m \in \End(E[m])$; if $\psi_m$
has characteristic polynomial $x^2 - t_mx + n_m$, then $t_m \equiv
\tr(\psi) \pmod{m}$. After computing $t \pmod{m}$ for each $m < L$, we
can compute $t \pmod{N}$ using the Chinese Remainder Theorem. The
bound in Lemma~\ref{L: boundfortrace} then lets us compute the
value of $\tr(\psi)$. Now, fix one such prime $m$.

\subsubsection{Computation of $\tr(\psi_m)$}\label{sec: explicit
  tmodn} %Now let $m$ be a prime different from $2, p$ \jenn{probably also $\ell$} such that
%$m \leq L$. We now outline the steps to computing
Let $t_m \equiv \tr(\psi) \bmod m$. Then the relation
$\psi_{m}^{2}-t_{m}\psi_{m} + n_{m} = 0$ holds in $\End(E[m]):=\End(E)/(m)$.
% (since $m$ is
%prime, it suffices to check its validity for a single nontrivial $P
%\in E[m]$).
%{\kirsten{Why is this true?}}\travis{what do we mean by check
%  validity? I don't understand the parenthetical comment} \jenn{I
%  think we should be able to just delete the parenthetical
%  comments. There are many other ways to prove this assertion and
%  none of them are interesting enough to include in a research paper
%  anyway.}
Here, $n_m \equiv \norm(\psi_m) = \ell^e \bmod m$, with
$0 \leq n_m < m$.

%Furthermore, using V\'elu's formula (\S \ref{sec:velusFormula}), one can compute the explicit isogeny $\psi: E \to E'$, then reduce modulo the ideal $I$ (using the notation in the discussion before Proposition \ref{P: reduction}) to obtain an explicit formula for $\psi_m$, whose degree is $O(m^2)$. Using the addition formula in \S \ref{ssec: addition formulas}, we may compute the explicit formula for $\psi_m^2 + n_m$, and reduce modulo $I$. Since $\psi$ has a quadratic minimal polynomial, one of these $\tau$ must satisfy $\psi_m^2 + \tau \psi_m + n_m \equiv 0 \bmod f_m.$

Furthermore, by the results in Section~\ref{sec:velusFormula}, one has
an explicit formula for $\psi_m: E[m] \to E[m]$ by reducing the
explicit coordinates for $\psi$ modulo the ideal $I$ (using the
notation in the discussion before Proposition \ref{P: reduction}),
with $\deg \psi_m = O(m^2)$. Using the addition formulas for $E$, we
can compute the explicit formula for $\psi_m^2 + n_m$, and reduce it
modulo $I$. The main modification to Schoof's algorithm, as it is
described in~\cite[5.1]{SS15}, is to replace the Frobenius
endomorphism on $E[m]$ with $\psi_m$. Having computed $\psi_m^2+n_m$
and $\psi_m$, for $\tau$ with $0\leq\tau \leq m-1$ we compute
$\tau \psi_m$ until
\[
\psi_m^2+n_m = \tau \psi_m 
\]
in $\End(E[m])$. Then $\tau = t_m$. Having computed $t_m$ for sufficiently many primes, we recover $\tr \psi$ using the Chinese remainder theorem. 

\subsubsection{Complexity analysis for computing the trace}

\begin{proposition} \label{P: isogenycomposition} Let $E/\F_q$ be a
  supersingular elliptic curve. Let $\psi$ be an isogeny of $E$ of
  degree $\ell^e$, given as a chain $\phi_1,\ldots,\phi_e$ of
  $\ell$-isogenies, whose explicit formulas are given by V\'elu's
  formula. The explicit formula for $\psi_m$ can be computed in
  $O(edM(d\log q)\log d)$ time, where $d \in \max \{\ell, m^2\}$.
\end{proposition}
\begin{proof}

The
  expression for $\psi_m$ can be computed by computing $(\phi_k)_m$
  for $k=1,\ldots,e$, composing the rational maps, and reducing modulo
  $I$ at each step. 
%Using na\"ive modular composition, and by
%  Proposition~\ref{P: isogenycomposition}, this takes
  %$O(edM(d\log q)\log d)=\tilde{O}(n^6)$ time.
The calculation of $f \circ g \bmod h$, where
$f, g, h\in \mathbb{F}_q[x]$ are polynomials of degree at most $d$,
takes $O(dM(d\log q))$ elementary operations using the na\"ive
approach.  Thus, computing $e$ of these
compositions, reducing modulo $f_m$ at each step, takes
$O(edM(d\log q) \log q)$ time. 
\end{proof}
%Additionally, for any
  %$\delta>0$, by~\cite[Corollary
  %7.2]{KU11}, 
  %\jenn{http://users.cms.caltech.edu/~umans/papers/KU08b.pdf, cor
    %5.2}, 
%the calculation of $f \circ g \bmod h$ takes $O(d^{1+\delta}\log q)$ elementary
  %operations.
%\kirsten{Isn't there a 1+delta in the exponent? Won't that
 % change things?} \travis{is it $O(d^2\log q)$ if we just compute f(g(x)) and reduce
 % mod $h(x)$?} \jenn{So the paper says that it can be done in exactly
 % $d^{1+\delta}\log d^{1+o(1)}$ computations (there is no big $O$),
  %for any $\delta>0$ In particular, if $\delta < \log_d N$ for some
  %constant $N$, isn't it already dominated by some constant term? I
  %thought having these negligible terms in the exponent was basically
  %replacing the suppression of constant terms in the big $O$
  %notation.}

%\jenn{Remark: This implies that we have a huge computational advantage if we can find cycles of short length (less than 6?!) So if we are able to consider a larger isogeny graph by considering several primes, then we might be able to save on the computational time. Of course, this means that we're looking at $m$-torsion points for composite $m$, and this might have an effect on what we can conclude about the trace mod $m$.}

We now wish to compute the trace of an endomorphism of $E$
corresponding to a cycle in $G(p,\ell)$. Since the diameter of
$G(p,\ell)$ is $O(\log p)$, we are interested in computing the trace
of a cycle of length $e=O(\log p)$ in $G(p,\ell)$. We are also
interested in the case where $\ell$ is a small prime, so we will take
$\ell=O(\log p)$. The resulting generalization of Schoof's algorithm
runs in time polynomial in $\log p$. We use $f(n)=\tilde{O}(g(n))$ to
mean that there exists $k$ such that $f(n)=O(g(n)\log^kn)$. 

%\travis{New statement/proof. Separate complexity of computing a chain
 % of isogenies, and then complexity of computing trace of an
  %endomorphism given as a chain of isogenies.}]
%We thank Andrew Sutherland for outlining the proof of the next lemma.
\begin{lemma}\label{lemma:ell-torsion}
Let $E/ \mathbb{F}_{p^2}$ be a supersingular elliptic curve, and let
$\ell \neq p$ be a prime.
Let $n = \max\{ \lceil \log p\rceil, \ell\}$. Then a basis for
$E[\ell]$ can be computed in expected $\tilde{O}(n^4)$ time. 
\end{lemma}
\begin{proof}
  Let $E/\F_{p^2}$ be given by a Weierstrass equation
  $y^2=x^3+Ax+B$. Let $f_{\ell}$ be the $\ell$-th division polynomial
  for $E$. We claim that $f_{\ell}$ has an irreducible factor
  $h_1(x)\in \F_{p^{2d}}[x]$ (where $d=1,2,3$ or $6$) of degree
  dividing $(\ell-1)/2$. This follows from the fact that any kernel
  polynomial of a degree $\ell$ isogeny $\phi: E\to E'$ will divide
  $f_{\ell}$, and $\phi$ is defined over $\F_{p^{2d}}$ for $d$ as
  above by
  Corollary~\ref{cor:isogenydefined}.  The
  degree of $h_1(x)$ will divide $(\ell-1)/2$, the degree of the
  kernel polynomial of $\phi$. We can find such an $h_1$ dividing
  $f_{\ell}$ in expected $\tilde{O}(n^4)$ time (see Algorithm 14.8 and
  Exercise 14.15 in~\cite{GJ13}). Set $K:=\F_{p^{2d}}[x]/(h_1(x))$ and
  let $L/K$ be a quadratic extension. We can compute a square root $b$
  of $x^3+Ax+B$ in $L$ in expected $\tilde{O}(n^4)$ time. Then
  $P_1:=(a,b)\in E[\ell]$ has order $\ell$.

To find another basis element, we first compute the kernel polynomial
$g(x)\in \F_{p^2}[x]$ of the isogeny with kernel $\langle P_1\rangle$. The time required for this step is
dominated by the time to compute $P_1$.
% First, compute [k]P_1 for k=1,\ldots,\ell-1. The kernel polynomial
% is the product of (x([k]P_1) for k=1,\ldots,\ell-1 (really only half
% of these) and we can compute this product using FFT.
Now we find another irreducible factor $h_2(x)$ of $f_{\ell}$ which is
coprime to $g_1(x)$, and repeat the process above to find another
point $P_2=(a_2,b_2)\in E[\ell]$. Then $P_2$ is necessarily
independent of $P_1$, because $a_2$ is not a root of $h_1$ and hence
$P_2$ is not in $\langle P_1 \rangle$. The pair $(P_1,P_2)$ is a basis
of $E[\ell]$.
\end{proof}

\begin{theorem}
  Let $E/\mathbb{F}_{p^2}$ be a supersingular elliptic and let
  $\ell\not=p$ be a prime. A sequence of $e$ many $\ell$-isogenies
  starting from $E$ can be computed in $\tilde{O}(n^5)$ time, where
  $n=\max\{ \lceil \log p\rceil, e, \ell\}$. The output is given as rational
  functions for each $\ell$-isogeny in the chain.
\end{theorem}
\begin{proof}
  To compute an $\ell$-isogeny originating from an elliptic curve $E$,
  we first compute $E[\ell]$ by computing the $\ell$-th division
  polynomial $f_{\ell}$ using a double-and-add approach. This
  computation takes $O(M(\ell^2\log p))=\tilde{O}(n^3)$ time
  (see~\cite[Section 5.1]{SS15}). By Lemma~\ref{lemma:ell-torsion} we
  can compute a basis of $E[\ell]$ in expected $\tilde{O}(n^4)$ time.
%  by finding two roots of $f_{\ell}$ and checking whether they are the
%  $x$-coordinates of a basis of $E[\ell]$; this process will take
%  expected time $\tilde{O}(n^4)$.  
  Now we can determine which subgroup of $E[\ell]$ corresponds to the
  next isogeny in the chain, if the path in $G(p, \ell)$ is specified
  by a sequence of $j$-invariants. We note that in this case the
  resulting isogeny is not determined by the $j$-invariants since
  there can be multiple edges in $G(p,\ell)$. We can compute the
  $\ell$-isogeny with kernel generated by $Q\in E[\ell]$ in
  $\tilde{O}(n^4)$ time by Proposition~\ref{P: veluisogeny}. We have
  to do $e$ many of these computations, which gives the total runtime
  of $\tilde{O}(n^5)$.
\end{proof}

\begin{theorem}
Let $\psi$ be an endomorphism of $E/\F_{p^2}$ given as a chain of
$\ell$-isogenies,
\[
\psi = \phi_e\circ \cdots \circ \phi_1,
\]
where each $\phi_k$ is specified by its rational functions and is
defined over $\F_q$. We note that we can take $\F_q$ to be an at most
degree $6$ extension of $\F_{p^2}$.  Let $n= \lceil\log p\rceil$ and
assume $e,\ell=O(n)$. Then  
the modified version of Schoof's algorithm computes $\tr \psi$ in
$\tilde{O}(n^7)$ time. 
\end{theorem}
\begin{proof}
 We follow the steps in our modification of Schoof's algorithm. Since
  $\norm \psi = \ell^e$, we first choose a
  bound $L = O(\log \ell^e)$. 
% and compute the division polynomials
 % $f_1, \ldots, f_{L}$. By Proposition~\ref{P: complexitydivpoly}, the
 % number of elementary operations needed to compute $f_1, \ldots, f_L$
  %is $O(L^5 M(\log q)) = O(e^5 (\log \ell)^5 M(\log q))$.

We can compute $\psi_m$ in time $\tilde{O}(n^6)$
time by Proposition~\ref{P: isogenycomposition}.
  For a prime $m < L$, we compute $\tr \psi_m$, the trace of the
  induced isogeny $\psi_m$ on $E[m]$, by reducing by the $m$-division
  polynomial $f_m$ whenever possible. 
%However, assuming \travis{KU assumptions},
  %$\psi_m$ can be computed in $O(e d^{1+\delta}(\log
  %q)^{o(1)})$ operations by~\cite{KU11}, where $d = e^2 \ell (\log
  %\ell)^2$, because $d \in \max \{O(\ell), O(m^2)\}$ and $m = O(\log
  %\ell^e)$.

Having computed $\psi_m$ and $\psi_m^2$, 
%using a modified version of Schoof's algorithm.
%described~\cite{SS15}. In particular, we replace $\pi_m$ and
%$\pi_m^2$, the map given by restricting the Frobenius endomorphism to
%$E[m]$, by $\psi_m$ and $\psi_m^2$. 
with the same argument as in the proof of Theorem~10 of~\cite{SS15}, we
can compute $t_m$ in $O((m+\log q)(M(m^2\log q)))$ time. This is
because once $\psi_m$ and $\psi_m^2$ are computed, the algorithm
proceeds the same way as Schoof's original algorithm.  We must repeat this $L=O(\log p) = O( n )$ times.

Once we compute $\tr \psi_m$ for each prime $m\not= p$ less than $L$,
we compute $\tr \psi$ using the Chinese Remainder Theorem.  This step
is dominated by the previous computations. Thus we have a total run
time of $\tilde{O}(n^7)$.
%(thus, $m = O(\log \ell^e)$). 
%Since $\psi$ is a composition of $e$
 % successive $\ell$-isogenies, computing the explicit formula for
  %these isogenies takes $O(\ell^6 M(\log q))$ elementary operations by
  %Proposition~\ref{P: veluisogeny}, so computing $e$ of these takes
  %$O(e \ell^6 M(\log q))$ elementary operations.
\end{proof}
\vskip 1cm
\newcommand{\etalchar}[1]{$^{#1}$}

%\bibliography{refrenceFile}

\begin{thebibliography}{ACC{\etalchar{+}}17}

\bibitem[ACC{\etalchar{+}}17]{SIKE}
Reza Azarderakhsh, Matthew Campagna, Craig Costello, Luca~De Feo, Basil Hess,
  Amir Jalali, David Jao, Brian Koziel, Brian LaMacchia, Patrick Longa, Michael
  Naehrig, Joost Renes, Vladimir Soukharev, and David Urbanik.
\newblock Supersingular isogeny key encapsulation.
\newblock Submission to the NIST Post-Quantum Standardization project, 2017.
\newblock
  \url{https://csrc.nist.gov/Projects/}\\\url{Post-Quantum-Cryptography/Round-1-Submissions}.

\bibitem[BJS14]{BJS14}
Jean-Fran\c{c}ois Biasse, David Jao, and Anirudh Sankar.
\newblock A quantum algorithm for computing isogenies between supersingular
  elliptic curves.
\newblock In {\em Progress in cryptology---{INDOCRYPT} 2014}, volume 8885 of
  {\em Lecture Notes in Comput. Sci.}, pages 428--442. Springer, Cham, 2014.

\bibitem[Cer04]{Cer14}
J.~M. Cervi{\~n}o.
\newblock Supersingular elliptic curves and maximal quaternionic orders.
\newblock In {\em Mathematisches {I}nstitut, {G}eorg-{A}ugust-{U}niversit\"at
  {G}\"ottingen: {S}eminars {S}ummer {T}erm 2004}, pages 53--60.
  Universit\"atsdrucke G\"ottingen, G\"ottingen, 2004.

\bibitem[CG14]{CG14}
Ilya Chevyrev and Steven~D. Galbraith.
\newblock Constructing supersingular elliptic curves with a given endomorphism
  ring.
\newblock {\em LMS J. Comput. Math.}, 17(suppl. A):71--91, 2014.

\bibitem[CGL09]{CLG2009}
Denis~X. Charles, Eyal~Z. Goren, and Kristin Lauter.
\newblock Cryptographic hash functions from expander graphs.
\newblock {\em J. Cryptology}, 22(1):93--113, 2009.

\bibitem[Deu41]{Deuring}
Max Deuring.
\newblock Die {T}ypen der {M}ultiplikatorenringe elliptischer
  {F}unktionenk\"orper.
\newblock {\em Abh. Math. Sem. Hansischen Univ.}, 14:197--272, 1941.

\bibitem[DFJP14]{DFJLP14}
Luca De~Feo, David Jao, and J\'er\^ome Pl\^ut.
\newblock Towards quantum-resistant cryptosystems from supersingular elliptic
  curve isogenies.
\newblock {\em J. Math. Cryptol.}, 8(3):209--247, 2014.

\bibitem[DG16]{GD16}
Christina Delfs and Steven~D. Galbraith.
\newblock Computing isogenies between supersingular elliptic curves over
  {$\Bbb{F}_p$}.
\newblock {\em Des. Codes Cryptogr.}, 78(2):425--440, 2016.

\bibitem[EHL{\etalchar{+}}18]{EHLMP}
Kirsten Eisentr{\"a}ger, Sean Hallgren, Kristin Lauter, Travis Morrison, and
  Christophe Petit.
\newblock Supersingular isogeny graphs and endomorphism rings: reductions and
  solutions.
\newblock {\em Eurocrypt 2018, LNCS 10822}, pages 329--368, 2018.

\bibitem[GPS17]{GPS2016}
Steven~D. Galbraith, Christophe Petit, and Javier Silva.
\newblock Identification protocols and signature schemes based on supersingular
  isogeny problems.
\newblock In Tsuyoshi Takagi and Thomas Peyrin, editors, {\em Advances in
  Cryptology -- ASIACRYPT 2017}, pages 3--33, Cham, 2017. Springer
  International Publishing.

\bibitem[KLPT14]{KLPT}
David Kohel, Kristin Lauter, Christophe Petit, and Jean-Pierre Tignol.
\newblock On the quaternion l-isogeny path problem.
\newblock {\em LMS Journal of Computation and Mathematics}, 17:418--432, 2014.

\bibitem[Koh96]{Kohel}
David Kohel.
\newblock {\em Endomorphism rings of elliptic curves over finite fields}.
\newblock PhD thesis, University of California, Berkeley, 1996.

\bibitem[LM04]{ML04}
Kristin Lauter and Ken McMurdy.
\newblock Explicit generators of endomorphism rings of supersingular elliptic
  curves.
\newblock Preprint, 2004.

\bibitem[McM14]{Mur14}
Ken McMurdy.
\newblock Explicit representation of the endomorphism rings of supersingular
  elliptic curves.
\newblock
  \url{https://phobos.ramapo.edu/~kmcmurdy/research/McMurdy-ssEndoRings.pdf},
  2014.

\bibitem[Mes86]{Mes86}
J.-F. Mestre.
\newblock La m\'ethode des graphes. {E}xemples et applications.
\newblock In {\em Proceedings of the international conference on class numbers
  and fundamental units of algebraic number fields ({K}atata, 1986)}, pages
  217--242. Nagoya Univ., Nagoya, 1986.

\bibitem[Neb98]{Neb98}
Gabriele Nebe.
\newblock Finite quaternionic matrix groups.
\newblock {\em Represent. Theory}, 2:106--223, 1998.

\bibitem[{NIS}16]{PQCS}
{NIST}.
\newblock Post-quantum cryptography, 2016.
\newblock \url{csrc.nist.gov/Projects/Post-Quantum-Cryptography}; accessed
  30-September-2017.

\bibitem[Piz80]{Piz80}
Arnold Pizer.
\newblock An algorithm for computing modular forms on {$\Gamma _{0}(N)$}.
\newblock {\em J. Algebra}, 64(2):340--390, 1980.

\bibitem[Sch85]{Schoof}
Ren\'e Schoof.
\newblock Elliptic curves over finite fields and the computation of square
  roots mod {$p$}.
\newblock {\em Math. Comp.}, 44(170):483--494, 1985.

\bibitem[Sch95]{Schoof95}
Ren\'e Schoof.
\newblock Counting points on elliptic curves over finite fields.
\newblock {\em J. Th\'eor. Nombres Bordeaux}, 7(1):219--254, 1995.
\newblock Les Dix-huiti\`emes Journ\'ees Arithm\'etiques (Bordeaux, 1993).

\bibitem[Sil09]{AEC}
J.H. Silverman.
\newblock {\em The Arithmetic of Elliptic Curves}.
\newblock Graduate Texts in Mathematics. Springer New York, 2009.

\bibitem[SS15]{SS15}
Igor~E. Shparlinski and Andrew~V. Sutherland.
\newblock On the distribution of {A}tkin and {E}lkies primes for reductions of
  elliptic curves on average.
\newblock {\em LMS J. Comput. Math.}, 18(1):308--322, 2015.

\bibitem[Sut13]{Sut13}
Andrew~V. Sutherland.
\newblock Isogeny volcanoes.
\newblock In {\em A{NTS} {X}---{P}roceedings of the {T}enth {A}lgorithmic
  {N}umber {T}heory {S}ymposium}, volume~1 of {\em Open Book Ser.}, pages
  507--530. Math. Sci. Publ., Berkeley, CA, 2013.

\bibitem[V{\'e}l71]{Vel71}
Jacques V{\'e}lu.
\newblock Isog\'enies entre courbes elliptiques.
\newblock {\em C. R. Acad. Sci. Paris S\'er. A-B}, 273:A238--A241, 1971.

\bibitem[Voi]{Voight}
John Voight.
\newblock {\em Quaternion Algebras}.
\newblock v.0.9.12, March 29, 2018.

\bibitem[vzGG13]{GJ13}
Joachim von~zur Gathen and J\"urgen Gerhard.
\newblock {\em Modern computer algebra}.
\newblock Cambridge University Press, Cambridge, third edition, 2013.

\bibitem[Wat69]{waterhouse}
William~C. Waterhouse.
\newblock Abelian varieties over finite fields.
\newblock {\em Ann. Sci. \'Ecole Norm. Sup. (4)}, 2:521--560, 1969.

\bibitem[YAJ{\etalchar{+}}17]{YAJJS}
Youngho Yoo, Reza Azarderakhsh, Amir Jalali, David Jao, and Vladimir Soukharev.
\newblock A post-quantum digital signature scheme based on supersingular
  isogenies.
\newblock In {\em Financial Cryptography and Data Security - 21st International
  Conference, {FC} 2017, Sliema, Malta, April 3-7, 2017, Revised Selected
  Papers}, pages 163--181, 2017.

\end{thebibliography}
%\bibliographystyle{alpha}	
%%%%%%%%%%%%%%%%%%%%%%%%%%%%%%%%%%%%%%%%%%%%%%%%%%%%%%%%%%%%%%%%%%%%%%%%%%%%%%%%%%%%%%%%%%%%%%%%%%%%%%%%%%%%	
\end{document}